\documentclass[11pt,oneside,english]{amsart}
\usepackage{subfiles, mathtools}
\usepackage{amssymb}
\usepackage[colorlinks]{hyperref}
\usepackage{graphicx}
\usepackage[pagewise]{lineno}
\usepackage{ulem}
\usepackage{float}
\usepackage{subcaption}

\makeatletter \theoremstyle{plain}
\newtheorem{thm}{Theorem}[section]
\newtheorem{lem}[thm]{Lemma}
\newtheorem{prop}[thm]{Proposition}

\numberwithin{equation}{section} 
\numberwithin{figure}{section} 
\theoremstyle{plain}
\theoremstyle{definition}

\newtheorem{rem}[thm]{Remark}
  \newtheorem{cor}[thm]{Corollary}

\newcommand{\bH}{{{\bf H}}}

\newcommand{\Aa}{{\mathcal {AA}}}

\newcommand{\C}{{{\mathbb C}}}

\newcommand{\R}{{{\mathbb R}}}

\newcommand{\arctanh}{\mathrm{arctanh}}
\newcommand{\e}{{\rm e}}

\usepackage{babel}

\makeatother

\begin{document}

\title[Sub-{R}iemannian geodesics in the affine-additive Group]{Sub-{R}iemannian geodesics in the affine-additive Group}

\author{Elia Bubani}
\author{Ioannis D. Platis}
\author{Georgios I. Simantiras}
\author{Dimitrios Tsolis}

\address{ }

\keywords{Pontryagin Maximum Principle, Time optimal synthesis, Carnot-Carath\'eodory spheres, Affine-additive group.
The first author was supported by the Swiss National Science Foundation, Grant Nr. 204501. The second author was supported by the Medicus program, Grant Nr. 83765.
}

\newcommand{\Addresses}{{
  \bigskip
  \footnotesize

  E.~Bubani (Corresponding author),  \textsc{University of Fribourg, Department of Mathematics,
Ch. du Musée 23,
1700 Fribourg,
Switzerland.}\par\nopagebreak
  \textit{E-mail address}, E.~Bubani: \texttt{elia.bubani@unifr.ch}

  \medskip

  I.D.~Platis, \textsc{Department of Mathematics, University of Patras, Panepistimioupolis, 26504 Rion, Achaia, Greece}\par\nopagebreak
  \textit{E-mail address}, I.D.~Platis: \texttt{idplatis@upatras.gr}

   \medskip

  G.I.~Simantiras, \textsc{Department of Mathematics, University of Patras, Panepistimioupolis, 26504 Rion, Achaia, Greece}\par\nopagebreak
  \textit{E-mail address}, G.I.~Simantiras: \texttt{gsimantiras@upatras.gr}

  \medskip
  D.~Tsolis, \textsc{Department of Mathematics, University of Patras, Panepistimioupolis, 26504 Rion, Achaia, Greece}\par\nopagebreak
  \textit{E-mail address}, D.~Tsolis: \texttt{d.tsolis@upatras.gr}

}}

\date{\today}

\begin{abstract}
We derive the sub-Riemannian geodesics of the affine-additive group using the Pontryagin Maximum Principle (PMP) and compute the associated cut times. We then proceed to construct the Carnot-Carath\'eodory sphere of the affine-additive group as well as the flask surface; the latter is the analogue of the Pansu sphere (bubble set) in the first Heisenberg group. 
\end{abstract}

\maketitle

\tableofcontents

\section{Introduction}
Sub-Riemannian geometry has evolved into an extensively studied domain with profound implications not only for pure mathematics and geometric analysis, but also for physics and the applied sciences. Sub-Riemannian manifolds serve highly effectively as models in numerous scientific subjects, including brain modeling and visual perception \cite{Hof-VP,H-VC,CS-RT,HP-RT,Citti-Cort}, quantum systems \cite{Boscain}, robot motion planning \cite{LSL-Robot, Mont}, and financial mathematics \cite{P-Fin,P-Fin2}. The fundamental characteristic of a sub-Riemannian manifold is that it does not allow motion in all directions. Instead, motion is strictly confined to a completely non-integrable horizontal distribution of the tangent bundle, that is, a fixed admissible subspace within any tangent space. The distance between two points is thus defined as the infimum of the lengths of these admissible paths. Assuming complete non-integrability for the distribution grants the distance to be finite, geodesic, and perfectly induces the manifold topology, as guaranteed by Chow's Theorem \cite[Chapter 2]{Mont}. Important classes of examples where these results apply include the Heisenberg group ${\mathbb H}$ and its higher-dimensional analogues \cite{CDPT}. For a broader geometric context, one might also refer to comprehensive treatises that bridge these structures with related fields like Cauchy-Riemann manifolds \cite{Dragomir} and complex hyperbolic geometry \cite{Goldman}.

The study of length-minimizing curves (sub-Riemannian geodesics) within this horizontal distribution constitutes the central challenge of the field. Finding sub-Riemannian geodesics is a variational problem which naturally lifts to the cotangent bundle, leading to intricate sets of Hamiltonian equations that govern the flow of horizontal trajectories. Since the foundational developments in the 1960s, the Pontryagin Maximum Principle (PMP, \cite{PBGM-PMP}) from optimal control theory has remained the principal framework for solving these equations.

In this work, we investigate the sub-Riemannian geometry of the affine-additive group, denoted by $\mathcal{AA}$. We model this three-dimensional Lie group as the product space $\mathcal{AA} := \mathbf{H}^1_{\mathbb{C}} \times \mathbb{R}$, where $\mathbf{H}^1_{\mathbb{C}} = \{(\lambda, t) : \lambda > 0, t \in \mathbb{R}\}$ denotes the hyperbolic plane. The group operation is defined by the law $$(\lambda, t, a) \cdot (\lambda', t', a') = (\lambda \lambda', \lambda t' + t, a + a'),$$ and we equip the space with the left-invariant contact 1-form $$\vartheta = da - \frac{dt}{\lambda}.$$ The horizontal distribution $\mathcal{H} = \ker\vartheta$ is completely non-integrable; hence the group is endowed with a proper contact sub-Riemannian manifold structure when equipped with a left-invariant metric that orthonormalizes the horizontal frame.

Three-dimensional contact Lie groups offer the most fundamental, non-trivial setting for studying sub-Riemannian phenomena. While the Heisenberg group--the prototypical example of a nilpotent Lie group--has been extensively documented in the literature, $\mathcal{AA}$ represents a fundamentally different geometric landscape. Crucially, the affine-additive group is non-unimodular; the trace of its adjoint representation does not vanish for every left-invariant vector field. This algebraic distinction translates into deep measure-theoretic and geometric divergences. For instance, $\mathcal{AA}$ possesses a non-vanishing 4-capacity at infinity, classifying it as a hyperbolic space, which implies it is not globally quasiconformally equivalent to the parabolic Heisenberg group, despite the two being globally contactomorphic~\cite{BBP1}.

A fascinating outcome of our study is the strong geometric parallel between horizontal projections of geodesics in $\mathcal{AA}$ and classical hyperbolic geometry. Specifically, the projections of sub-Riemannian extremals onto $\mathbf{H}^1_{\mathbb{C}}$ correspond intimately to hyperbolic geodesics and horocycles. We construct the Carnot-Carath\'eodory (CC) unit sphere of the affine-additive group, as well as a compact topological structure we refer to as the \textit{flask surface}. This surface encapsulates the fundamental continuous symmetries of the sub-Riemannian distance and illustrates the rigid geometric obstructions inherent to non-unimodular spaces; all in all, it serves as the analogue of the Pansu sphere (bubble set) in the first Heisenberg group (see \cite{Pansu1,Pansu2}). Keeping in mind Theorem \ref{Thm: cut-time}, one may derive an analogous statement to Theorem \ref{thm: CCSphere} for general CC-spheres of radius $r>0$, further expanding upon the literature surrounding CC-spheres and sub-Finsler structures \cite{BR-CCmetrics,BBDS-Finslerspheres}. We underline a quite striking geometric disparity: while in the sub-Riemannian Heisenberg group, CC spheres of radius $r\neq1$ are neatly recovered by anisotropic dilations of factor $r$ with respect to the unit CC sphere \cite{Monti-CCballs}, the profile of $\mathcal{S}_\Aa(\e,r)$ in the affine-additive group may vary without a particular geometric pattern by changing $r>0$. To visualise this contrast, we present plots for $\mathcal{S}_\Aa(\e,1)$ and $\mathcal{S}_\Aa(\e,5)$.

To derive sub-Riemannian geodesics, we deploy the standard Hamiltonian formalism of the PMP. Because $\mathcal{AA}$ is a contact three-dimensional manifold, non-trivial abnormal extremals are ruled out, ensuring that all geodesics are determined precisely as the projections of normal extremal curves from the cotangent bundle. We calculate the exact analytic expressions for these trajectories and, consequently, study their properties. A critical component of this work involves the characterization of the cut time: this is when Pontryagin extremals cease to be geodesics. This is accomplished using an extended Hadamard technique which has already been described and utilized in \cite{Sachkov,ABB,AS-Engel,AH-Cartan}. Essentially, we first locate Maxwell points (those are points where distinct geodesics of equal length intersect) and then identify conjugate points via the vanishing of the Jacobian of the sub-Riemannian exponential map. We conclude using Hadamard theorem proving the correspondence between conjugate times and cut times.
\medskip

This paper is organized as follows.
In Section \ref{sec-prel}, we provide preliminaries for our study starting from describing the sub-Riemannian structure on $\mathcal{AA}$ and detailing its left-invariant vector fields. Section \ref{sec-geod} is devoted to the Time Optimal Synthesis: that is, the explicit construction of the full set of geodesics (Theorem \ref{thm: Geodesics of AA}) together with their cut time (Theorem \ref{Thm: cut-time}). This relies first on computing the Hamiltonian equations from the PMP and secondly on applying the Hadamard technique with an educated guess of the cut times. In Section \ref{sec-ccflask}, we focus in surfaces constructed by geodesics of $\Aa$, culminating in the derivation of the CC-unit sphere (Theorem \ref{thm: CCSphere}) and the Flask Surface (Propositions \ref{prop: flask 1} and \ref{Prop: flask 2}).
\medskip

\noindent{\textbf{Acknowledgements}} The authors would like to thank Zolt\'an M. Balogh (Bern, Switzerland), Enrico Le Donne (Fribourg, Switzerland) and Nicola Paddeu (Fribourg, Switzerland) for fruitful discussions on the subject of this paper. The authors claim that no LLM has been used in this paper except for pictures \ref{pic: CC-sphere r=1} and \ref{pic: CC-sphere r=5}. We used the assistance of ChatGPT-5.6 Sol for writing a code then submitted to Wolfram Mathematica to produce such plots.

\section{Preliminaries}\label{sec-prel}
\subsection{The affine-additive group \texorpdfstring{$\Aa$}{}}

The \textit{affine-additive group} $\mathcal{AA}$ (for a detailed description, see \cite{Bubani},\cite{BPPT}) is defined as $\mathcal{AA}:=\mathbf{H}^1_{\mathbb{C}} \times \mathbb{R}$, where $\mathbf{H}^1_\C:=\{(\lambda,t):\lambda>0,t\in\R\}$ denotes the hyperbolic plane with its Riemannian metric given by $g_{\text{hyp}}:=\frac{d\lambda^2+dt^2}{\lambda^2}$. The group law is given by
$$
(\lambda,t,a) \cdot (\lambda',t',a')=(\lambda \lambda' , \lambda t'+t,a+a'),
$$
the neutral element is denoted by ${\rm e}=(1,0,0)$ and the inverse element is $(\lambda,t,a)^{-1}=(1/\lambda,-t/\lambda,-a)$.
Let $p=(\lambda,t,a)\in\Aa$, a basis for the tangent space ${\rm T}_p\,\mathcal{AA}$ comprises the left-invariant vector fields
$$U=\lambda \partial_\lambda, \, V=\lambda \partial_t+\partial_a, \, W =\partial_a.$$
The only non-vanishing Lie bracket relation is $[U,V]=V-W$. 
We equip $\Aa$ with the contact form $$\vartheta =da-\frac{dt}{\lambda},$$ so that $\ker\vartheta=\text{span}\{U,V\}$ and $W$ is the Reeb vector field. We therefore define the \textit{horizontal distribution} to be $\mathcal{H}_{\Aa}:=\ker\vartheta$. The corresponding coframe of $\{U,V,W\}$ is given by $$\omega_1= \frac{d\lambda}{\lambda}, \; \omega_2=\frac{dt}{\lambda}, \; \vartheta=da-\frac{dt}{\lambda}.$$

\subsection{Sub-Riemannian distance on \texorpdfstring{$\Aa$}{}}
We begin this section by describing the sub-Riemannian distance on $\Aa$ in detail.
An absolutely continuous curve $\gamma:[c,d]\to\Aa$, $\gamma(s)=(\lambda(s),t(s),a(s)),$ $\,s\in[c,d]$, is \textit{horizontal} when $\dot\gamma(s)\in\ker\vartheta_{\gamma(s)}$ for almost every $s\in[c,d]$. This is equivalent to the o.d.e.
\begin{equation}\label{horiz ode}
    \dot a(s)-\frac{\dot t(s)}{\lambda(s)}=0, \text{ a.e. } s\in[c,d].
\end{equation}
We denote by $g_{\Aa}(\cdot,\cdot)$ the sub-Riemannian metric on $\mathcal{H}_{\Aa}$ which makes $\{ U, V\}$ an orthonormal basis. In this way, the horizontal velocity $|\dot\gamma|_H$ of $\gamma$ is defined by the relation
\begin{equation}\label{horiz speed}
|\dot\gamma|_H:=\left(g_{\Aa}(\dot\gamma,U_\gamma)^2+g_{\Aa}(\dot\gamma,V_\gamma)^2\right)^{1/2}=\frac{\sqrt{\dot \lambda^2+\dot t^2}}{\lambda}.
\end{equation}
Let $\pi:\Aa\to\bH^1_\C$ be the canonical projection given by $\pi(\lambda,t,a)=(\lambda,t),\,(\lambda,t,a)\in\Aa$. The horizontal length of $\gamma$ is
\begin{equation}\label{lhyphor}
\ell(\gamma):=\int_c^d |\dot\gamma|_H(s)\,ds= \ell_{\text{hyp}}(\pi(\gamma)),
\end{equation}
where $\ell_{\text{hyp}}$ denotes the hyperbolic length of the curve $\pi(\gamma)\subset\bH_\C^1$ with respect to the metric $g_{\text{hyp}}$.

The corresponding sub-Riemannian or Carnot-Carath\'eodory distance $d_{\Aa}$ associated to the sub-Riemannian metric $g_{\Aa}(\cdot,\cdot)$ is defined for all $p, q  \in\Aa$ as follows:
\begin{equation*}
    d_{\Aa}(p,q)=\inf_{\gamma\in \Gamma_{p,q}}\{\ell(\gamma) \},
\end{equation*}
where $\Gamma_{p,q}$ is the following family of curves:
$$
\Gamma_{p,q}=\{\gamma:[0,1]\to\Aa \text{ horizontal and such that}\;
       \gamma(0)=p,\;
       \gamma(1)=q\}.
$$
We underline here that since $\mathcal{H}_\Aa$ is bracket generating, the distance $d_{\Aa}$ is finite, geodesic and
induces the manifold topology (see Montgomery \cite[Chapter 2]{Mont}).\\
The isometry group with respect to $d_\Aa$ is denoted with ${\rm Isom}_{d_{\Aa}}(\mathcal{AA})$ and consists of $\Aa\times\mathbb{Z}_2$: it is generated by left translations $L_{(\lambda_0, t_0,a_0)}:\mathcal{AA}\to\mathcal{AA}$ with respect to $(\lambda_0, t_0,a_0)\in\Aa$ and
the reflection $j:\Aa \to \Aa$ given by
\begin{eqnarray*}
j(\lambda ,t,a) =(\lambda ,-t,-a). 
\end{eqnarray*}
A horizontal curve $\gamma:[c,d]\to\Aa$ is a \textit{geodesic} if 
\begin{equation*}
    d_\Aa(\gamma(s_1),\gamma(s_2))=|s_1-s_2|,\quad\forall s_1,s_2\in[c,d].
\end{equation*}
\subsection{PMP and cut times via Hadamard technique}\label{Preliminaries}
This section contains the main tools and results for our work. Before introducing Pontryagin Maximum Principle (PMP) we need to recall basic Hamiltonian formalism. For a point $p=(\lambda,t,a)\in\Aa$ and a covector $\xi=(\xi_1,\xi_2,\xi_3)\in{\rm T}^*_p \Aa$, canonical coordinates $(p,\xi)$ on ${\rm T}^*\Aa$ are defined in such a way that every element $\mu\in{\rm T}^*\Aa$ is written as 
 \begin{equation}\label{eq: mu}
     \mu=\xi_1\,d\lambda\rvert_{p}+\xi_2\,dt\rvert_{p}+\xi_3\,da\rvert_{p}.
 \end{equation}
 In the previous section we considered a sub-Riemannian structure on $\Aa$ with respect to the horizontal distribution $\mathcal{H}_\Aa=\text{span}\{U,V\}$. The fiberwise-linear functions $h_i:{\rm T}^*\Aa\to\R$, $i\in\{1,2\}$, associated to $U$ and $V$ are given by
\begin{equation}\label{eq: h1, h2 def}
    h_1(\mu):=\langle\mu,U_p\rangle,\quad
    h_2(\mu):=\langle\mu,V_p\rangle,
\end{equation}
where $\mu$ is identified as in \eqref{eq: mu} and  $\langle\cdot,\cdot\rangle$ denotes the standard pairing between tangent vectors and their covectors. The \textit{sub-Riemannian Hamiltonian} $H:{\rm T}^*\Aa\to\R$ is given by 
\begin{equation*}
    H(\mu)=\frac{1}{2}\left(h_1(\mu)^2+h_2(\mu)^2\right)
\end{equation*}
The functions $h_i$, $i\in\{1,2\}$ are associated with the \textit{Hamiltonian vector fields} $\overrightarrow{h_i}$, $i\in\{1,2\}$ as follows: we make use of canonical coordinates $(p,\xi)$ for $\mu\in{\rm T}^*\Aa$, then
\begin{equation*}
    \overrightarrow{h_i}=\frac{\partial h_i}{\partial{\xi_1}}\,\partial_\lambda+\frac{\partial h_i}{\partial{\xi_2}}\,\partial_t+\frac{\partial h_i}{\partial{\xi_3}}\,\partial_a-\frac{\partial h_i}{\partial{\lambda}}\,\partial_{\xi_1}-\frac{\partial h_i}{\partial{t}}\,\partial_{\xi_2}-\frac{\partial h_i}{\partial{a}}\,\partial_{\xi_3}, \quad i\in\{1,2\}.
\end{equation*}
Pontryagin Maximum Principle gives the necessary condition satisfied by geodesics in the sub-Riemannian geometry of $\Aa$ (see for example \cite[Theorem 12.10]{AS-CT} or \cite[Theorem 4.20]{ABB}). Curves satisfying such condition are called \textit{extremal trajectories}.
\begin{thm}[Characterization of Pontryagin extremals]\label{Thm: PMP}
    Let $\gamma:[0,T]\to\Aa$ be a geodesic and let $\bar{u}(\cdot)=(\bar{u_1}(\cdot),\bar{u_2}(\cdot))$ be the corresponding minimal control, i.e., for a.e. $s\in[0,T]$,
    \begin{equation*}
        \dot{\gamma}(s)=\bar{u_1}(s)U_{\gamma(s)}+\bar{u_2}(s)V_{\gamma(s)},\quad \ell(\gamma)=\int_0^T \sqrt{\bar{u_1}^2(s)+\bar{u_2}^2(s)}\,ds=d_\Aa(\gamma(0),\gamma(T)),
    \end{equation*}
    with $\sqrt{\bar{u_1}^2(s)+\bar{u_2}^2(s)}=1$ a.e. on $[0,T]$. Then there exists a curve $\mu:[0,T] \to{\rm T}^*\Aa$ with initial data $\mu(0)$ identified with $(\gamma(0),\xi(0))$, such that
    \begin{equation}\label{eq: Hamilton eq}
        \dot{\mu}(s)=\sum_{i=1}^2\bar{u_i}(s)\overrightarrow{h_i}(\mu(s)),\quad\text{a.e. }s\in [0,T],
    \end{equation}
and such that one of the following conditions is satisfied:
\begin{align}
    h_i(\mu(s))&=\bar{u_i}(s), \;\forall s\in[0,T],\;\forall i\in\{1,2\}, \label{eq: Normal extr}\tag{N}\\
    h_i(\mu(s))&=0, \;\forall s\in[0,T],\;\forall i\in\{1,2\} \label{eq: Abnormal extr}\tag{A}.
\end{align}
Moreover, in case \eqref{eq: Abnormal extr} one has $\mu(s)\neq0$ for all $s\in[0,T]$.
\end{thm}
\begin{rem}
An extremal trajectory $\gamma:[0,T]\to\Aa$ which satisfies condition \eqref{eq: Normal extr} (respectively \eqref{eq: Abnormal extr}) is called \textit{normal trajectory} (respectively \textit{abnormal trajectory}). If an extremal trajectory satisfies only \eqref{eq: Abnormal extr}, then it is said to be a \textit{strictly abnormal trajectory}.
\end{rem}
We are going to present the sub-Riemannian exponential map and then Maxwell, conjugate and cut times associated to an extremal trajectory. For these topics we refer the reader to  \cite[Chapter 8]{ABB} and \cite{Sachkov}. \\
We denote $\R_+=[0,\infty)$ and recall that an extremal trajectory is a curve $\gamma:\R_+\to\Aa$ with initial data $\gamma(0)={\rm e}$ and $\xi(0)\in{\rm T}^*_{\rm e}\Aa$ which satisfies conditions given in Theorem \ref{Thm: PMP}. In this way the dependence of $\gamma$ with respect to $\xi(0)$ is adopted with the subscript $\gamma_{\xi(0)}$. \\ We consider the set 
\begin{equation*}
    C_{\rm e}:=\left\{ \xi(0)\in {\rm T}_{\rm e}^*\Aa \,: \, H(\xi(0))=\frac{1}{2}\right\},
\end{equation*} 
and denote by $\mathrm{exp}_{\rm e}:C_{\rm e}\times\R_+\to \Aa$, the \textit{sub-Riemannian exponential map} given by:
\begin{equation*}
     \left(\xi(0), s \right)\mapsto \gamma_{\xi(0)}(s),
\end{equation*}
where $\gamma$ is the extremal trajectory with initial data $\gamma(0)={\rm e}$ and $\xi(0)\in{\rm T}^*_{\rm e}\Aa$. \\
A point $p=\gamma(T)\in\Aa$, with $T>0$, of an extremal trajectory $\gamma(\cdot)=\mathrm{exp}_{\rm e}(\xi(0),\cdot):\R_+\to\Aa$ is called a \textit{Maxwell point along} $\gamma$, if there is another extremal trajectory $\tilde{\gamma}(\cdot)=\mathrm{exp}_{\rm e}(\tilde{\xi(0)},\cdot):\R_+\to\Aa$ such that  $\tilde{\gamma}(s)\neq \gamma(s)$ for all $s\in (0,T)$ and $p=\tilde{\gamma}(T)=\gamma(T)$. The first instant where the latter condition happens is called \textit{first Maxwell time} and is denoted with $T^1_{\rm Max}$. The set of all Maxwell points of $\Aa$ is called \textit{the Maxwell set} and we write it as $\mathrm{Max}$. Following \cite[Section 4.6]{Sachkov}, it is usually hard to describe all Maxwell points and to determine the first of
them. To this end, a \textit{symmetry} is defined to be a pair of mappings $\varepsilon:C_{\rm e}\times\R_+\to C_{\rm e}\times\R_+$ and $\sigma:\Aa\to\Aa$ such that 
\begin{equation}\label{eq: symmetry}
    \sigma\circ\exp_{\rm e}=\exp_{\rm e}\circ\,\varepsilon.
\end{equation}
We are going to explain the role of symmetries in determining Maxwell times in Section \ref{sec: cut-time analysis}.

 A point $p \in \Aa$ is called \textit{conjugate to} ${\rm e}\in\Aa$ if there exists a time $s>0$ and an initial covector $\xi(0)$ such that $p=\mathrm{exp}_{\rm e}(\xi(0),s)$ and $(\xi(0),s)$ is a critical point of $\mathrm{exp}_{\rm e}$. The specific time that it happens is called \textit{conjugate time} and is denoted with $T_{\rm conj}$.\\
 The \textit{cut time} of an extremal trajectory  $\gamma:\R_+\to\Aa $ is defined by \begin{align*}
        T_{\text{cut}} &=\sup \{ T>0: \gamma:[0,T]\to\Aa  \text{ is a geodesic}  \}.
    \end{align*}
    If $T_{\text{cut}}=+\infty$, we say that $\gamma$ has no cut time.
Now we relate the previous concepts with the fact that a normal trajectory cannot be a geodesic after a Maxwell time (see \cite[Lemma 3]{Sachkov}), thus 
\begin{equation*}
    T_{\mathrm{cut}}\leq T_{\mathrm{Max}}^1.
\end{equation*}
In order to determine cut times of extremal trajectories we shall follow the guide of the extended Hadamard technique from \cite[Section 13.4]{ABB}. We present the main step of such guide below after giving a definition.\\
We consider smooth manifolds $X,Y$. We omit definitions of proper and covering maps which can be found in \cite[Chapter 13]{ABB}. A continuous map $F:X\to Y$ is called \textit{finite-sheeted covering map} if it is a covering map such that for all $y\in Y$ the cardinality $|F^{-1}(y)|$ of the fiber $F^{-1}(y)$ is finite (see \cite[Section 1.3]{Hatcher}). If $Y$ is connected, the cardinality of every fiber is constant and it is called number of sheets of $F$. A basic fact (see \cite[Proposition 1.32]{Hatcher}) tells that the number of sheets corresponds to the index $[\pi_1(Y):F_*(\pi_1(X))]$, where $F_* : \pi_1(X)\to \pi_1(Y)$ is the induced homomorphism.\\ The observation that the number of sheets is constant along connected components of the target manifold combined with \cite[Theorem 13.23]{ABB} gives the following.
\begin{thm}\label{Hadamard}
Let $X,Y$ be smooth manifolds and let $F: X \to Y$ be a smooth map such that the following hold 
     \begin{enumerate}
        \item $\dim X =\dim Y,$ 
        \item $X,Y$ are connected,
        \item $F$ is a local diffeomorphism,
        \item $F$ is proper,
    \end{enumerate}
    then $F$ is a finite-sheeted covering map.
\end{thm}
\section{Sub-Riemannian geodesics of \texorpdfstring{$\Aa$}{}}\label{sec-geod}
\subsection{Application of PMP}\label{sec: PMP}
This section concerns sub-Riemannian geodesics in $\Aa$. We present first some important facts which set up the operative perspective to apply Theorem \ref{Thm: PMP}.\\
All normal trajectories are locally geodesics (see \cite[Theorem 17.2]{AS-CT}). Further, the contact structure of $\Aa$ excludes the possibility of strictly abnormal trajectories (see \cite[Section 7.3.3]{D-MLG}) and, thanks to \cite[Lemma 4.43]{ABB}, one may show the nonexistence of non-constant general abnormal trajectories in $\Aa$. In what follows we focus on the study of normal trajectories inspired by \cite[Section 4.4.2]{ABB}.

Guided by Theorem \ref{Thm: PMP}, we consider the curve $\gamma(\cdot)=(\lambda(\cdot),t(\cdot),a(\cdot)):[0,T]\to \Aa$ to be a geodesic with associated curve of covectors $\xi(\cdot)=(\xi_1(\cdot),\xi_2(\cdot),\xi_3(\cdot)):[0,T]\to{\rm T}_{\gamma(\cdot)}^*\Aa$ and set up initial data of $\gamma$ and $\xi$ as:
\begin{equation}\label{eq: initial data}
    \gamma(0)={\rm e},\quad \xi(0)=(\eta_0,\tau_0,\xi_0)\in{\rm T}^*_{\rm e}\Aa.
\end{equation}

We recall that $\gamma$ is assumed to be parametrized by arc-length (i.e. $\sqrt{\bar{u_1}(s)^2+\bar{u_2}(s)^2}=1$ a.e. on $[0,T]$). In this manner condition \eqref{eq: Normal extr} implies $h_1^2(\mu(s))+h_2^2(\mu(s)) = 1$ on $[0,T]$. We define the \textit{angular function} $\phi:[0,T]\to\R$ such that 
\begin{equation}\label{eq: phi}
    h_1=\cos \phi (s) \text{ and } h_2=\sin \phi (s), \quad\forall s\in  [0,T]. 
\end{equation}
Substituting \eqref{eq: phi} into \eqref{eq: Hamilton eq} yields 
\begin{equation}\label{eq: Ham AA}
    \begin{aligned}
        \dot{\gamma}(s) &= \cos\phi(s) U_{\gamma(s)} + \sin \phi(s) V_{\gamma(s)},\quad \forall s\in(0,T)\\
    \dot{\xi}(s) &= -(\cos\phi(s)\xi_1(s)+\sin\phi(s)\xi_2(s)){d\xi_1}, \quad\forall s\in(0,T) .
    \end{aligned}
\end{equation}
We rewrite equations \eqref{eq: Ham AA} obtaining the system
\begin{align}\label{eq: ode_system}
\begin{cases}
    \dot \lambda(s) = \lambda(s) \cos \phi(s)\\
    \dot t(s) = \lambda(s) \sin \phi(s)\\
    \dot a(s) = \sin \phi(s)\\
    \dot{\xi_1}(s)=-\cos\phi(s)\xi_1(s)-\sin\phi(s)\xi_2(s)\\
    \dot{\xi_2}(s)=0\\
    \dot{\xi_3}(s)=0
\end{cases}\hspace{-.4cm},\quad s\in(0,T).
\end{align}
Recalling definitions for $h_i$, $i\in\{1,2\}$, given in \eqref{eq: h1, h2 def}, we have 
\begin{equation}\label{eq: h1, h_2 applied}
    h_1=\xi_1\lambda \text{ and } h_2=\xi_3+\xi_2\lambda.
\end{equation}
Thus, condition $h_1^2+h_2^2=1$, combined with the initial data \eqref{eq: initial data}, imposes that $\xi(0)$ satisfies 
\begin{equation*}
    \eta_0^2 + (\xi_0+\tau_0)^2 = 1.
\end{equation*}
We are now in position to state and prove the following theorem describing the sub-Riemannian geodesics of $\Aa$.
\begin{thm}\label{thm: Geodesics of AA}
Let \eqref{eq: initial data} be the initial data, with $\eta_0^2 + (\xi_0+\tau_0)^2 = 1$. The sub-Riemannian geodesics $\gamma:[0,T]\to\Aa$ are given by the following cases:
\begin{enumerate}
    \item if $\tau_0=0$ and $\eta_0=0$, then $\xi_0=\pm 1$ and the geodesic is the straight line $\gamma(s)=(1, \xi_0 s,\xi_0 s)$;
    \item if $\tau_0=0$ and $\eta_0\neq 0$, then $\gamma(s)=\left(e^{\eta_0 s}, \frac{\xi_0}{\eta_0}\left(e^{\eta_0s}-1\right),\xi_0 s\right)$;
    \item if $\tau_0\neq0$, then $\gamma(s)=(\lambda(s),t(s),a(s))$ is given by:
     \begin{equation}
        \begin{aligned}\label{AAGeod3}
        \lambda(s) &= \frac{\sin\phi(s)-\xi_0}{\tau_0},\\
        t(s) &= \frac{\eta_0-\cos\phi(s)}{\tau_0},\\
        a(s) &= \phi(s) - \phi_0 + \xi_0 s,
    \end{aligned}
    \end{equation}
    where $\phi(0)=\phi_0 $ such that $\cos\phi_0 = \eta_0$ and $\sin\phi_0 = \xi_0+\tau_0$, and $\phi:[0,T]\to\R$ is a smooth function which satisfies the o.d.e. $\dot{\phi}(s) = \sin\phi(s) - \xi_0,\,\forall s\in(0,T)$.
    
    Explicitly, \begin{itemize}
    \item if $|\xi_0| < 1$ and $\xi_0\neq 0$, then
    \begin{align}\label{eq: phi 1}
        \phi(s) &= 2\arctan \left( \frac{1}{\xi_0} + \frac{\sqrt{1-\xi_0^2}}{\xi_0} \tanh \left( \frac{(s+c_1)\sqrt{1-\xi_0^2}}{2} \right) \right), 
    \end{align}
  \item if $\xi_0= 0$, then
 \begin{align}\label{eq: phi 2}
     \phi (s) &=2 \arctan \left( \pm e^s c_1' \right),
 \end{align}

    \item if $\xi_0 < -1$, then
    \begin{align}\label{eq: phi 3}
        \phi(s) &= 2\arctan \left( \frac{1}{\xi_0} - \frac{\sqrt{\xi_0^2-1}}{\xi_0} \tan \left( \frac{(s+c_2)\sqrt{\xi_0^2-1}}{2} \right) \right) ,
    \end{align}
    \item if $\xi_0 = -1$, then
    \begin{align}\label{eq: phi 4}
        \phi(s) &= 2\arctan \left( -\frac{2+s+c_3}{s+c_3} \right),
    \end{align}
\end{itemize}
for some constants $c_1, c_1', c_2, c_3 \in \mathbb{R}$ given by
 \begin{align*}
    c_1 &= \frac{2}{\sqrt{1-\xi_0^2}} \arctanh\left( \frac{ \xi_0(\xi_0+\tau_0)-(1+\eta_0) }{ (1+\eta_0)\sqrt{1-\xi_0^2} } \right),\\
    c_1'&=\ln \left| \frac{\xi_0+\tau_0}{1+\eta_0} \right| ,\\
    c_2&=\frac{2}{\sqrt{\xi_0^2-1}} \left( \arctan\left( \frac{ 1+\eta_0-\xi_0(\xi_0+\tau_0) }{ (1+\eta_0)\sqrt{\xi_0^2-1} } \right) \right), \\
    c_3&= -\frac{2(1+\eta_0)} {\eta_0+\tau_0}.
\end{align*}
\end{enumerate}
The cases for $\xi_0 \geq 1$ are recovered by the action of the reflection $j:\Aa\to\Aa$.
\end{thm}

\begin{proof}
We begin our proof by looking at system \eqref{eq: ode_system} which we rewrite as follows
\begin{align}\label{eq: ode_system_1}
\begin{cases}
    \dot \lambda(s) = \lambda(s) \cos \phi(s),\\
    \dot t(s) = \lambda(s) \sin \phi(s),\\
    \dot a(s) = \sin \phi(s),\\
    \dot{\xi_1}(s)=-\cos\phi(s)\xi_1(s)-\tau_0\sin\phi(s),\\
    \xi_2(s)=\tau_0,\\
    \xi_3(s)=\xi_0,
\end{cases}\hspace{-.4cm},\quad s\in(0,T).
\end{align}
First we consider the case $\tau_0=0$. The fifth and sixth equations of system \eqref{eq: ode_system_1}, together with definitions \eqref{eq: h1, h_2 applied}, imply that 
\begin{equation}\label{eq: h2}
    h_2(s)=\xi_0 , \quad\forall s\in[0,T].
\end{equation}
If also $\eta_0=0$, the assumption on the initial data $\xi(0)$ yields $\xi_0=\pm 1$. Thus \eqref{eq: phi} imposes that 
\begin{equation*}
    \sin\phi(s)=\xi_0,\,\cos\phi(s)=0, \quad\forall s \in[0,T],
\end{equation*}
and system \eqref{eq: ode_system_1} gives the curve 
\begin{equation*}
    \gamma(s)=(1, \xi_0 s,\xi_0 s),\quad s\in[0,T]. 
\end{equation*}
If $\eta_0\neq0$, the initial data conditions for $\gamma(0)$ and $\xi(0)$ together with \eqref{eq: h2} imply that 
\begin{equation}\label{eq: h1}
    h_1(s)=\eta_0 ,\quad \forall s\in[0,T].
\end{equation}
Thanks to \eqref{eq: phi} and \eqref{eq: h1}, the fourth equation of system \eqref{eq: ode_system_1} becomes the o.d.e.
\begin{equation}\label{eq: ode xi_1}
    \begin{aligned}
  \dot{\xi_1}(s)&=-\eta_0\,\xi_1(s),\quad s\in(0,T),\\
  \xi_1(0)&=\eta_0,
\end{aligned}
\end{equation}
which gives 
\begin{equation*}
    \xi_1(s)=\eta_0 e^{-\eta_0 s},\quad s\in[0,T].
\end{equation*}
The latter expression combined with \eqref{eq: h1, h_2 applied} and \eqref{eq: h1} gives that
\begin{equation}\label{eq: lambda deg}
    \lambda(s)=e^{\eta_0 s},\quad s\in[0,T].
\end{equation}
Using \eqref{eq: phi}, \eqref{eq: h1}, \eqref{eq: h2} and \eqref{eq: lambda deg} we integrate the second and third equations from system \eqref{eq: ode_system_1} and obtain
\begin{equation*}
    \gamma(s)=\left(e^{\eta_0 s},\frac{\xi_0}{\eta_0}(e^{\eta_0s}-1),\xi_0 s\right), \quad s\in[0,T].
\end{equation*}
Now we assume $\tau_0\neq0$. From \eqref{eq: phi}, \eqref{eq: h1, h_2 applied}, the first and the last three equations of system \eqref{eq: ode_system_1} we obtain the o.d.e.
\begin{equation}\label{eq: ode phi}
    \dot{\phi}(s)=\sin\phi(s)-\xi_0,\quad s\in (0,T),
\end{equation} 
where $\phi(0)=\phi_0$ satisfies
\begin{equation*}
    \cos\phi_0=\eta_0 \text{ and } \sin\phi_0=\xi_0+\tau_0.
\end{equation*}
 Thus, integrating $\dot{a}(s) = \sin\phi(s) = \dot{\phi}(s) + \xi_0$, we have $a(s) = \phi(s) - \phi_0 + \xi_0 s$.

For $\lambda(s)$, notice from \eqref{eq: ode_system_1} that:
$$ \frac{\dot{\lambda}(s)}{\lambda(s)} = \cos\phi(s) = \frac{\dot{\phi}(s)\cos\phi(s)}{\sin\phi(s) - \xi_0}. $$
Integrating the latter with respect to $s\in[0,T]$ yields 
 \begin{equation}\label{eq: lambda}
     \lambda(s) = \frac{\sin\phi(s) - \xi_0}{\sin\phi_0 - \xi_0} = \frac{\sin\phi(s) - \xi_0}{\tau_0},\quad s\in[0,T].
 \end{equation}
For $t(s)$, we recall \eqref{eq: h1, h_2 applied} and system \eqref{eq: ode_system_1} and get:
$$ \dot{t}(s) = \lambda(s)\,\sin\phi(s) = \frac{(\sin\phi(s)-\xi_0)\sin\phi(s)}{\tau_0} = \frac{\dot{\phi}(s)\sin\phi(s)}{\tau_0}. $$
Integrating with respect to $s\in[0,T] $ gives $t(s) = \frac{\cos\phi_0 - \cos\phi(s)}{\tau_0}$. To sum up, for $\tau_0\neq0$ the geodesic $\gamma$ is given by
\begin{equation}\label{eq: gamma tau not 0}
    \gamma(s)=\left(\frac{\sin\phi(s) - \xi_0}{\tau_0},\frac{\cos\phi_0 - \cos\phi(s)}{\tau_0},\phi(s) - \phi_0 + \xi_0 s\right),\quad s\in[0,T].
\end{equation}
We are going to solve the o.d.e. \eqref{eq: ode phi} which we rewrite as \begin{equation}\label{eq: angle ode} 
\frac{\dot{\phi}(s)}{\sin\phi(s) - \xi_0} = 1, \quad s \in(0,T).
\end{equation} Using the half-angle substitution $x(s)=\tan\left(\frac{\phi(s)}{2}\right)$ and the associated identities relating the parameter $x$ with $\sin \phi$ and $\cos \phi$, the o.d.e. \eqref{eq: angle ode} becomes:
\begin{equation}\label{eq: phi4}
   -\frac{2}{\xi_0} \frac{\dot{x}(s)}{\left( x(s)-\frac{1}{\xi_0} \right)^2+1-\frac{1}{\xi_0^2}}=1,\quad s\in(0,T).
\end{equation}
Without loss of generality we assume $\tau_0>0$. We observe that the o.d.e. \eqref{eq: ode phi} and \eqref{eq: lambda} show that $\tau_0>0$ is equivalent to $\dot{\phi}(s)>0$, $\forall s\in(0,T)$.
We recall that assumption $\eta_0^2 + (\xi_0 + \tau_0)^2 = 1$ gives $|\xi_0 + \tau_0|\leq 1$ and since $\tau_0 > 0$, we must have $\xi_0 < 1$. Thus, we consider three distinct ranges for $\xi_0$: $\xi_0 \in (-1, 1)$, $\xi_0 < -1$, and $\xi_0 = -1$.
Solving the o.d.e. \eqref{eq: phi4} yields the explicit solutions \eqref{eq: phi 1}, \eqref{eq: phi 2},\eqref{eq: phi 3} and \eqref{eq: phi 4}.

In the same spirit, if we assume $\tau_0<0$, it guarantees $\dot{\phi}(s)<0$, $\forall s\in(0,T)$. Thus, we must have $\xi_0>-1$ splitting into three distinct cases for $\xi_0$: $\xi_0\in (-1,1),\, \xi_0>1$ and $\xi_0=1$. Working similarly as in case $\tau_0>0$, one can find that the solution is the reflection of the trajectory $\gamma(s)$ given by \eqref{eq: gamma tau not 0}, i.e. the curve $j(\gamma(s))=\left(\lambda(s),-t(s),-a(s) \right)$.
\end{proof}
\begin{rem}
It is worth pointing out the continuous dependence of the geodesics with respect to the initial data $\xi(0)=(\eta_0,\tau_0,\xi_0)$. If we consider case (2) and let $\eta_0 \to 0$, we get $\xi_0 \to \pm 1$ and recover the straight lines of case (1). Similarly, in case (3), taking the limit as $\tau_0 \to 0$ gives the o.d.e. \eqref{eq: ode xi_1}. As a consequence, if $\eta_0 \ne 0$ we obtain case (2) and if $\eta_0 = 0$ we find again case (1). 
\end{rem}
\begin{rem}\label{rem: proj}
If $\tau_0 \ne 0$, the projection $\pi(\gamma)\subset\mathbf{H}^1_{\mathbb{C}}$ of the geodesic $\gamma$ is a Euclidean circle centered at $\left(-\frac{\xi_0}{\tau_0}, \frac{\eta_0}{\tau_0}\right)$ with radius $\frac{1}{|\tau_0|}$.
\end{rem}

\subsection{Analysis of cut times}\label{sec: cut-time analysis}
In this section we are going to determine cut times for extremal trajectories of the form $\gamma:\R_+\to\Aa$ with initial data $\gamma(0)={\rm e}$ and $\xi(0)\in{\rm T}^*_{\rm e} \Aa$. We use relations \eqref{eq: phi} and \eqref{eq: h1, h_2 applied} at time $t=0$ to rewrite $\xi(0)$ with respect to the coframe $\{\omega_1,\omega_2,\vartheta\}|_\e$ as follows:
\begin{equation}\label{eq: xi0 coframe}
    \xi(0)=\cos\phi_0\,\omega_1|_\e+\sin\phi_0\,\omega_2|_\e+\xi_0\,\vartheta|_\e.
\end{equation}
This consideration allows us to see $C_\e$ as the cylinder $C_\e\cong\R\times\mathbb{S}^1$ and  identifies initial covectors $\xi(0)$ with $(\xi_0,\phi_0)\in\R\times \mathbb{S}^1$.  According to this notation the exponential map becomes $\exp_\e:\R\times\mathbb{S}^1\times\R_+\to\Aa$ with
\begin{equation}\label{Expe}
    \exp_\e(\xi_0,\phi_0,s)=(\lambda(s),t(s),a(s)), \quad(\xi_0,\phi_0,s)\in\R\times\mathbb{S}^1\times\R_+,
\end{equation}
where the coordinate functions $(\lambda,t,a)$ are given by Theorem \ref{thm: Geodesics of AA}.\\
The next result concerns how to determine Maxwell times following the technique explained in \cite[Section 3.1.1]{Maxstrata}. Before proceeding, we need to define first the concept of period. Let $\xi(0)$ as in \eqref{eq: xi0 coframe} with $\sin\phi_0-\xi_0\neq 0$, the \textit{period} $T_{\xi_0}$ of a solution $\phi(s)$ which satisfies the o.d.e. \eqref{eq: ode phi} as
\begin{equation}\label{eq : Txi0}
    T_{\xi_0}:=\inf\left\{s>0 \, : \, \phi(s)-\phi_0=2k\pi, \quad \text{for some} \,k\in \mathbb{Z}\setminus\{0\}\right\}.
\end{equation}
In other words, we can say that $T_{\xi_0}$ is the first positive time such that $\phi(s)$ completes a full rotation.
If no such time exists, we say that the angular function $\phi$ has infinite period, i.e. $T_{\xi_0}=\inf\emptyset=+\infty$.\\
The period is a natural candidate for the first Maxwell time $T_{\mathrm{Max}}^1$ which also depends on $\xi(0)$. We also observe (see \cite[Section 1.3]{chic}) that definition \eqref{eq : Txi0} is equivalent to the following:
\begin{equation}\label{eq : Txi_02}
    T_{\xi_0}=\int_{0}^{2\pi}\frac{du}{|\sin u-\xi_0|},
\end{equation}
for non-equilibrium solutions, i.e. when $\dot{\phi}(0)\neq 0\iff\sin\phi_0-\xi_0\neq 0$.

In what follows we are going to relate $T_{\xi_0}$ with the Maxwell time $T^1_{\rm Max}$. 
\begin{lem}
    Let $\gamma=\gamma_{\xi(0)}:\R_+\to\Aa$ be an extremal trajectory with initial data $\gamma(0)=\e$ and $(\xi_0,\phi_0)\in\R\times\mathbb{S}^1$. Then $\gamma(T_{\xi_0})$ is a Maxwell point along $\gamma.$
\end{lem}
\begin{proof}
    Note that for $|\xi_0|>1$, $T_{\xi_0}$ is the time for \eqref{eq : Txi0} to advance $2\pi$ i.e. $\phi (T_{\xi_0})=\phi_0+2\pi.$ Using \eqref{AAGeod3} we observe 
    \begin{equation}\label{qpoint}
    \gamma(T_{\xi_0})=(1,0,2\pi+\xi_0T_{\xi_0}).
    \end{equation}
    Consider a second extremal trajectory $\tilde{\gamma}(s)=j(\gamma(T_{\xi_0})^{-1}\cdot\gamma(T_{\xi_0}-s))$ from $\e$ with initial covector $(\xi_0,\pi-\phi(T_{\xi_0})).$ Using \eqref{AAGeod3}, one can see that $\gamma(s)$ is not identically equal to $\tilde{\gamma}(s)$, but $\gamma(T_{\xi_0})=\tilde{\gamma}(T_{\xi_0}).$ Hence $\gamma (T_{\xi_0})$ is a Maxwell point. 
\end{proof}
As a consequence of the previous result, $T_{\xi_0}$ provides an upper bound for the first Maxwell time:
\begin{equation}\label{MAXbound}
T_{\mathrm{Max}}^1\leq T_{\xi_0}.
\end{equation}

\begin{prop}\label{lem: ang period}
    Let $\gamma=\gamma_{\xi(0)}:\R_+\to\Aa$ be an extremal trajectory with initial data $\gamma(0)=\e$ and $(\xi_0,\phi_0)\in\R \times\mathbb{S}^1$. The period of the angular function is
    \begin{equation}
    T_{\xi_0}=
        \begin{cases}
          \frac{2\pi}{\sqrt{\xi_0^2-1}}, \quad \text{if } |\xi_0|>1,\\
          +\infty, \quad \text{if } -1\leq \xi_0\leq1,
        \end{cases}\label{period}
    \end{equation}  
    Moreover, $T^1_{\rm Max}=T_{\xi_0}$.
\end{prop}

\begin{proof}
Without loss of generality, we assume $\sin\phi_0-\xi_0>0$. Indeed, one may recover the case $\sin\phi_0-\xi_0<0$ by means of the reflection $j$. We consider the following distinction:\\
\noindent{} \textit{Case I:} If $\xi_0<-1$, by the proof of Theorem \ref{thm: Geodesics of AA}, the following conditions are satisfied:
    $$\dot{\phi}(s)=\sin\phi(s)-\xi_0, \qquad \dot{\phi}(0)=\tau_0.$$

Since $\sin\phi(s)-\xi_0\geq-1-\xi_0>0$, we find that $\dot{\phi}(s)>0$, for all $s>0$. Therefore, $\phi(s)$ is strictly-increasing. Using \eqref{eq : Txi_02} and the half-angle substitution $z=\tan\frac{u}{2}$, we calculate:
$$T_{\xi_0}=\frac{2\pi}{\sqrt{\xi_0^2-1}}.$$

\noindent{} \textit{Case II:} If $-1\leq \xi_0\leq1$, the equation $\sin\phi=\xi_0$, has a solution $\phi^*$ between the interval $(\phi_0, \phi_0+2\pi)$. Let $S$ be the time such that $\phi(S)=\phi_0+2\pi$. Due to continuity of $\phi(s)$, there is an $s^*\in (0,S)$, such that $\phi(s^*)=\phi^*$. Hence, it is $\dot{\phi}(s^*)=0$.

In addition, the constant function $\phi(s)=\phi^*$ is also a solution of the o.d.e. \eqref{eq: angle ode}. Therefore, the uniqueness of solutions implies that they must coincide on their common interval of definition, which leads to a contradiction, since $\dot{\phi}(0)=\tau_0>0$. Hence, $\phi(s)$ cannot reach $\phi^*$ in finite time. Consequently,
$$\phi_0<\phi(s)<\phi^*<\phi_0+2\pi, $$
for every finite $s>0$. Then, $\phi$ cannot complete a full rotation in finite time and combining with \eqref{eq : Txi0}, it is $T_{\xi_0}=+\infty$.

Finally, if $\sin\phi_0-\xi_0=0$, then the o.d.e. \eqref{eq: ode phi} reduces to $\dot{\phi}(s)=0$. Thus $\phi(s)\equiv \phi_0$ is a solution. Then, again by \eqref{eq : Txi0}, we obtain $T_{\xi_0}=+\infty$.

\end{proof}

The next result concerns the Maxwell set, we focus on the Maxwell points obtained by applying the symmetries of $\Aa$ on the extremal trajectories. This procedure is inspired by \cite[Section 3.2]{Maxstrata}.
\begin{prop}\label{Max1}
Let $\gamma=\gamma_{\xi(0)}:\R_+\to\Aa$ be an extremal trajectory with initial data $\gamma(0)=\e$ and $(\xi_0,\phi_0)\in\R\times\mathbb{S}^1$.
Then the set
$$\mathcal{L}=\left\{(1,0,a)\in \Aa \, : \, a\neq0 \right\}$$
corresponds to the Maxwell set ${\rm Max}$.
\end{prop}
\begin{proof}
     For $\xi_0<-1$, the Maxwell point obtained by a first full rotation, is of the form \eqref{qpoint}. We set $a_{-}(\xi_0):=2\pi+\xi_0\, T_{\xi_0}=2\pi\left(1+\frac{\xi_0}{\sqrt{\xi_0^2-1}} \right)$. Since $$\lim_{\xi_0\to -\infty}\left(a_{-}(\xi_0)\right)=0, \qquad\lim_{\xi_0\to -1^-}\left(a_{-}(\xi_0)\right)=-\infty,$$
    the range of $a_{-}(\xi_0)$ is $(-\infty,0)$. We work similarly for $\xi_0>1$. Thus the inclusion $\mathcal{L}\subset \mathrm{Max}$ follows.

    Now we proceed proving $ \mathrm{Max}\subset \mathcal{L}$. Let $T>0$ and let $(\xi_0,\phi_0,T)\in C_\e\times \R_+$ such that
\begin{equation}\label{eq: gamma Max}
    \gamma(s)=\exp_{\e}(\xi_0,\phi_0,s)=\left(\lambda(s),t(s),a(s)\right), \quad s\in [0,T],
\end{equation}    
    is the solution from Theorem \ref{thm: Geodesics of AA}, where $\hat{p}=\gamma(T)$ denotes its endpoint and $\phi(s)$ is the corresponding angular solution. 
    In what follows we construct pairs $(\varepsilon_i,\sigma_i), \, i\in\{1,2,3\}$ that are non-trivial symmetries with the purpose of determining $\rm Max$.\\
    We define the mappings $\varepsilon_i : C_{\e}\times \R_+\to C_\e\times \R_+, \, i\in\{1,2,3\},$
    given by
    \begin{align}
        \varepsilon_1(\xi_0,\phi_0,T)&=(-\xi_0,-\phi_0,T),\label{e1}\\ \varepsilon_2(\xi_0,\phi_0,T)&=(-\xi_0,\phi(T)+\pi,T),\quad (\xi_0,\phi_0,T)\in C_\e\times\R_+. \notag\\ \varepsilon_3(\xi_0,\phi_0,T)&=(\xi_0,\pi-\phi(T),T),\notag
    \end{align}
Further, the mappings $\sigma_i : \Aa\to \Aa, \, i\in\{1,2,3\},$ are given by $$\sigma_1(p)=j(p), \quad \sigma_2(p)=p^{-1}, \quad \sigma_3(p)=j(p^{-1}),\quad p\in\Aa.  $$
Recalling the curve $\gamma$ given in \eqref{eq: gamma Max}, one can verify that the curves
$$\gamma_1(\cdot)=j(\gamma(\cdot)), \quad \gamma_2(\cdot)=\gamma(T)^{-1}\cdot \gamma(T-\cdot), \quad \gamma_3(\cdot)=j(\gamma_2(\cdot)),$$
correspond to the initial covectors $(-\xi_0,-\phi_0), \, (-\xi_0,\phi(T)+\pi), \, (\xi_0,\pi-\phi(T))$, respectively. Therefore, by the uniqueness of the solutions, we write
\begin{align*}
    \gamma_1(s)&=\exp_\e(-\xi_0,-\phi_0,s)\\
    \gamma_2(s)&=\exp_\e(-\xi_0,\phi(T)+\pi,s)\\
    \gamma_3(s)&=\exp_\e(\xi_0,\pi-\phi(T),s).
\end{align*}
Computing at $s=T>0$, we obtain
$\gamma_i(T)=\sigma_i(\gamma(T))$ for $i\in\{1,2,3\}$. In other words,
$$\exp_\e\left(\varepsilon_i(\xi_0,\phi_0,T)\right)=\sigma_i\left(\exp_\e(\xi_0,\phi_0,T)\right), \quad i\in\{1,2,3\}.$$
Since $T>0$ was assumed arbitrary, the pairs $(\varepsilon_i,\sigma_i), \, i\in\{1,2,3\}$ satisfy the symmetry definition (\ref{eq: symmetry}). 

Now, for some $T_0>0$, a Maxwell point $p_0=\gamma(T_0)$ is generated by $(\varepsilon_i,\sigma_i), \, i\in\{1,2,3\}$ and it must satisfy 
\begin{equation*}\label{max123}
\varepsilon_i(\xi_0,\phi_0,T_0)\neq (\xi_0,\phi_0,T_0) \quad \text{and} \quad \sigma_i(p_0)=p_0,\;  \forall i\in\{1,2,3\}.
\end{equation*}
We study each case separately:

    \noindent{} \textit{Case I:} We have
    $$\gamma_1(T_0)=\gamma(T_0)\iff (\lambda(T_0),-t(T_0),-a(T_0))=(\lambda(T_0),t(T_0),a(T_0))\iff t(T_0)=a(T_0)=0,$$
which, after combining the last fact with the second and the third o.d.e of system (\ref{eq: ode_system}), leads to $\sin\phi(s)\equiv 0,$ for all $s\in (0,T_0)$. Therefore, $\phi(s)\equiv 0$ or $\phi(s)\equiv \pi$ for all $s\in (0,T_0)$, which yields $\gamma_1(s)=\gamma(s)=(e^{\pm s},0,0)$, for all $s\in (0,T_0)$, thus contradicting the definition of a Maxwell point. 

\noindent{} \textit{Case II:} It holds
$$\gamma_2(T_0)=\gamma(T_0)\iff (\lambda(T_0)^{-1},-t(T_0)\cdot\lambda(T_0)^{-1},-a(T_0))=(\lambda(T_0),t(T_0),a(T_0)).$$
Thus $\lambda(T_0)=1, \, t(T_0)=0, \, a(T_0)=0$. Following the same argument of the first case, we have $\gamma(s)=(e^{\pm s},0,0)$, which leads to a contradiction since $\gamma(T_0)=(e^{\pm T_0},0,0)\neq (1,0,0)$.

\noindent{} \textit{Case III:} For the third symmetry, we have 
$$\gamma_3(T_0)=\gamma(T_0)\iff (\lambda(T_0)^{-1},t(T_0)\cdot \lambda(T_0)^{-1},a(T_0))=(\lambda(T_0),t(T_0),a(T_0)).$$
The last condition holds if and only if $\lambda(T_0)=1$. Assuming that $\sin\phi_0-\xi_0\neq 0$, by the expression of $\lambda(s)$ in (\ref{AAGeod3}), we obtain
$$\phi(T_0)=\phi_0+2k\pi \quad \text{or} \quad\phi(T_0)=\pi-\phi_0+2k\pi, \quad k\in\mathbb{Z}.$$
If $k=0$ we do not get any Maxwell points. If $k\neq0$ we have $\gamma_3(s)=\gamma(s)$ for all $s\in (0,T_0).$ Then $\phi(T_0)=\phi_0+2k\pi, \quad k\in \mathbb{Z}\setminus\{0\}$. Therefore $T_0=2k \,T_{\xi_0}=\frac{2k\pi}{\sqrt{\xi_0^2-1}}>0, \, k\in\mathbb{Z}\setminus\{0\}$ is equivalent to $ |\xi_0|>1$, which produces Maxwell points of the form $(1,0,a), \, a\neq0$, i.e. $(1,0,a)\in\mathcal{L}$.\\ 
If $\sin\phi_0-\xi_0=0$, using the solutions of Theorem \ref{thm: Geodesics of AA} again, we do not have Maxwell points.

Combining Cases $I-III$, we find that the non-trivial symmetries $(\varepsilon_i,\sigma_i), \, i\in\{1,2,3\}$ generate Maxwell points on $\mathcal{L}$. In other words, $\mathrm{Max}\subset \mathcal{L}$, which completes the proof.
\end{proof}
Before proceeding with the next result we collect some useful computations in the following 
\begin{lem}\label{lem: J}
    Let $\exp_{\e}:C_\e\times\R_+\to\Aa$ be represented in $(\xi_0,\phi_0,s)$-coordinates as in \eqref{Expe}. If $\sin\phi_0-\xi_0\neq0$, then the Jacobian determinant $J$ of $
D\,\exp_\e\big|_{(\xi_0,\phi_0,s)}:{\rm T}_{(\xi_0,\phi_0,s)}C_\e\times \R_+\to{\rm T}_{\gamma(s)}\Aa
$, computed with respect to the bases $\left\{ \partial_{\xi_0},\partial_{\phi_0},\partial_s\right\}\rvert_{(\xi_0,\phi_0,s)}$ and $\left\{U,V,W\right\}\rvert_{\gamma(s)}$, is given by 
\begin{equation*}
 J=\frac{
2\sin\frac{\phi(s)-\phi_0}{2}
}{
(\sin\phi(s)-\xi_0)
(\sin\phi_0-\xi_0)
(\xi_0^2-1)
}
\left(
2\sin\frac{\phi(s)-\phi_0}{2}
+
s\left(
\xi_0\cos\frac{\phi(s)-\phi_0}{2}
-
\sin\frac{\phi(s)+\phi_0}{2}
\right)
\right).   
\end{equation*}
\end{lem}
\begin{proof}
We let $(\xi_0,\phi_0,s)\in C_\e\times\R_+$ and use Theorem \ref{thm: Geodesics of AA} and the o.d.e. (\ref{eq: angle ode}) to calculate the arc-length $s$ of $\gamma(s)=\exp_\e(\xi_0,\phi_0,s)$:
\begin{equation}
s
=
\int_{\phi_0}^{\phi(s)}
\frac{du}{\sin u-\xi_0}. \label{arclength}
\end{equation}
Furthermore, we have:
$a(s)=\int_{\phi_0}^{\phi(s)}\frac{\sin u}{\sin u-\xi_0}\,du$, $s\in\R_+$. From now on, we consider $\lambda, t, a$ and $\phi$ as an independent endpoint variables. The chain-rule corrections are proportional to
$\frac{\partial\exp_\e}{\partial s}$ and therefore do not change the
determinant. Partial derivatives of $\exp_e(\xi_0,\phi_0,s)=(\lambda,t,a)$ are computed as follows 
\begingroup
\small
\begin{alignat*}{3}
\frac{\partial\lambda}{\partial\xi_0}
&=
\frac{\sin\phi-\sin\phi_0}
     {(\xi_0-\sin\phi_0)^2},
&\qquad
\frac{\partial\lambda}{\partial\phi_0}
&=
\frac{\cos\phi_0(\xi_0-\sin\phi)}
     {(\xi_0-\sin\phi_0)^2},
&\qquad
\frac{\partial\lambda}{\partial s}
&=
\frac{\sin\phi-\xi_0}{\sin\phi_0-\xi_0}\cos\phi,
\\
\frac{\partial t}{\partial\xi_0}
&=
\frac{\cos\phi_0-\cos\phi}
     {(\xi_0-\sin\phi_0)^2},
&\qquad
\frac{\partial t}{\partial\phi_0}
&=
\frac{\xi_0\sin\phi_0+\cos\phi\cos\phi_0-1}
     {(\xi_0-\sin\phi_0)^2},
&\qquad
\frac{\partial t}{\partial s}
&=
\frac{\sin\phi-\xi_0}{\sin\phi_0-\xi_0}\sin\phi,
\\
\frac{\partial a}{\partial\xi_0}
&=
\int_{\phi_0}^{\phi}
\frac{\sin u}{(\sin u-\xi_0)^2}\,du,
&\qquad
\frac{\partial a}{\partial\phi_0}
&=
-\frac{\sin\phi_0}{\sin\phi_0-\xi_0},
&\qquad
\frac{\partial a}{\partial s}
&=
\sin\phi.
\end{alignat*}
Now, the desired identity for $J$ follows from a direct computation. 
\endgroup

\end{proof}

\begin{prop}\label{conj}
    The conjugate time of a trajectory $\gamma:\R_+\to\Aa$ with initial data $\gamma(0)=\e$ and $(\xi_0,\phi_0)\in\R\times\mathbb{S}^1$ is given by
$$
T_{\mathrm{conj}}
=
\begin{cases}
\displaystyle
\frac{2\pi}{\sqrt{\xi_0^2-1}},
&\text{if } |\xi_0|>1,\\[1.2ex]
+\infty,
& \text{if }-1\leq\xi_0\leq1.
\end{cases}
$$
\end{prop}
\begin{proof}
By definition, for a conjugate time $s>0$ we have that the tangent map of the exponential map
$
D\,\exp_\e\big|_{(\xi_0,\phi_0,s)}:{\rm T}_{(\xi_0,\phi_0,s)}C_\e\times \R_+\to{\rm T}_{\gamma(s)}\Aa
$ is not surjective.
Following the reasoning from \cite[Section 3.1.3]{BR-CCmetrics}, we have that the latter condition is equivalent to the linear dependence of the vectors
$\gamma(s)^{-1}\cdot
\frac{\partial\exp_\e}{\partial\xi_0},
\quad    
\gamma(s)^{-1}\cdot
\frac{\partial\exp_\e}{\partial\phi_0},
\quad
\gamma(s)^{-1}\cdot
\frac{\partial\exp_\e}{\partial s},
$
with $\gamma(s)=\exp_\e(\xi_0,\phi_0,s)$. Hence, the conjugate times are determined by
$
 J=0,
$
where $J$ is the Jacobian determinant of the matrix representing $D\exp_\e\big|_{(\xi_0,\phi_0,s)}$ with respect to the bases $\left\{ \partial_{\xi_0},\partial_{\phi_0},\partial_s\right\}\rvert_{(\xi_0,\phi_0,s)}$ and $\left\{U,V,W\right\}\rvert_{\gamma(s)}.$

Thus, the conjugate times are the positive solutions of
\begin{equation}\label{eq: J=0}
     J=0 .
\end{equation}
We now assume $\sin\phi_0-\xi_0>0$, we shall consider the cases $\sin\phi_0-\xi_0<0$ and $\sin\phi_0-\xi_0=0$ at the end of the proof.

\noindent \textit{Case I:} $\xi_0<-1$. We let
$
\kappa=\sqrt{\xi_0^2-1}$ and
$
\Theta=\frac{\kappa s}{2}.
$
Using the half-angle substitution
$z=\tan\frac{u}{2}$ in (\ref{arclength}),
we obtain
$$
s
=
\int_{\tan(\phi_0/2)}^{\tan(\phi/2)}
\frac{2\,dz}{2z-\xi_0(1+z^2)}.
$$
Since
$
2z-\xi_0(1+z^2)
=
-\frac{\kappa^2+(1-\xi_0z)^2}{\xi_0},
$
we find
$$
s
=
\frac{2}{\kappa}
\left[
\arctan\left(
\frac{1-\xi_0\tan(u/2)}{\kappa}
\right)
\right]_{u=\phi_0}^{u=\phi}.
$$
Therefore,
\begin{equation}\label{eq: Theta}
    \Theta
=
\arctan\left(
\frac{1-\xi_0\tan(\phi/2)}{\kappa}
\right)
-
\arctan\left(
\frac{1-\xi_0\tan(\phi_0/2)}{\kappa}
\right) +k\pi,\quad k\in\mathbb{Z}.
\end{equation}
When
\begin{equation}\label{eq: tat hyp}
    1+\left(
\frac{1-\xi_0\tan(\phi/2)}{\kappa}
\right)\left(
\frac{1-\xi_0\tan(\phi_0/2)}{\kappa}
\right)\neq0,
\end{equation}
we rewrite \eqref{eq: Theta} using the identity
$
\tan(\arctan x-\arctan y+k\pi)
=
\frac{x-y}{1+xy},
$ for all $k\in\mathbb Z$ and for all $x,y\in\R$ such that $1+xy\neq 0$,
resulting in
\begin{equation*}
\tan\Theta
=
\frac{
\kappa
\left(
\tan\frac{\phi}{2}-\tan\frac{\phi_0}{2}
\right)
}{
\tan\frac{\phi}{2}+\tan\frac{\phi_0}{2}
-
\xi_0
\left(
1+\tan\frac{\phi}{2}\tan\frac{\phi_0}{2}
\right)
}.
\end{equation*} 
By means of trigonometric identities, $\tan\Theta$ simplifies to
$\tan\Theta
=
\frac{
\kappa\sin\left(\frac{\phi-\phi_0}{2}\right)
}{
\sin\left(\frac{\phi+\phi_0}{2}\right)
-
\xi_0\cos\left(\frac{\phi-\phi_0}{2}\right)
}.
$
We rewrite the latter identity as
\begin{equation}\label{eq: J1}
\sin\left(\frac{\phi-\phi_0}{2}\right)
=
\frac{\tan\Theta}{\kappa}
\left(
\sin\left(\frac{\phi+\phi_0}{2}\right)
-
\xi_0\cos\left(\frac{\phi-\phi_0}{2}\right)
\right).
\end{equation}
Also, we observe that
\begin{align}\label{eq: J2}
(\sin\phi-\xi_0)(\sin\phi_0-\xi_0)
&=
\sin\phi\sin\phi_0
-\xi_0(\sin\phi+\sin\phi_0)
+\xi_0^2
\nonumber\\
&=
\left(
\sin\frac{\phi+\phi_0}{2}
-
\xi_0\cos\frac{\phi-\phi_0}{2}
\right)^2
+
(\xi_0^2-1)
\sin^2\frac{\phi-\phi_0}{2}.
\end{align}
Since $\kappa^2=\xi_0^2-1$ and
$s=\frac{2\Theta}{\kappa},
$
we have
\begin{align}\label{eq: J3}
&2\sin\frac{\phi-\phi_0}{2}
+
s\left(
\xi_0\cos\frac{\phi-\phi_0}{2}
-
\sin\frac{\phi+\phi_0}{2}
\right)
=
\frac{2}{\kappa}
\left(
\sin\frac{\phi+\phi_0}{2}
-
\xi_0\cos\frac{\phi-\phi_0}{2}
\right)
(\tan\Theta-\Theta).
\end{align}
Substituting \eqref{eq: J1}, \eqref{eq: J2} and \eqref{eq: J3} into the expression for $J$ given in Lemma \ref{lem: J}, yields:
\begin{equation*}\label{eq: newJ}
     J
=
\frac{4}{(\xi_0^2-1)^2}
\sin\Theta
\bigl(
\sin\Theta-\Theta\cos\Theta
\bigr).
\end{equation*}
From a straightforward computation we deduce that assumption \eqref{eq: tat hyp} does not hold if and only if $ \cos\Theta=0$. In this case we have $\sin\Theta=\pm1$ and combined with the continuity of $J$, leads to $J=\frac{4}{(\xi_0^2-1)^2}>0
$.

To sum up, we have
$
 J=0$ if and only if $
\sin\Theta=0
\text{ or }
\tan\Theta=\Theta.
$
The first positive zero of $\sin\Theta$ is $\Theta=\pi$. By the definition of $\Theta$ and $\kappa$
this yields the candidate conjugate time:
$$
s=\frac{2\pi}{\kappa}=\frac{2\pi}{\sqrt{\xi_0^2-1}}.
$$
It remains to verify that the second factor does not vanish at an earlier
time. Since $\Theta(s)$ is strictly increasing and $\Theta(0)=0$, we have
$$
0<s<\frac{2\pi}{\kappa}
\quad\Longleftrightarrow\quad
0<\Theta<\pi.
$$
We define the function
$
g(\Theta)=\sin\Theta-\Theta\cos\Theta.$
For $\Theta\in (0,\pi)$, we have
$
g'(\Theta)=\Theta\sin\Theta>0,
$
 $\Theta\in(0,\pi)$.
Since $g(0)=0$, it follows that
$
g(\Theta)>0$
for all $\Theta\in (0,\pi)$.
Moreover, $\sin\Theta>0$ for all $\Theta\in(0,\pi)$. Therefore,
$
 J>0
\text{ for all } s\in \left(0, \frac{2\pi}{\sqrt{\xi_0^2-1}}\right),
$
implying that no conjugate time occurs before $\Theta=\pi$. Consequently,
$\Theta=\pi$ gives the first solution of \eqref{eq: J=0}, and
$T_{\mathrm{conj}}(\xi_0,\phi_0)
=
\frac{2\pi}{\sqrt{\xi_0^2-1}},$
for $
\xi_0<-1.$

\noindent \textit{Case II:} $|\xi_0|<1$. Recalling that we assumed 
$\sin\phi_0>\xi_0$, we let
$\mu=\sqrt{1-\xi_0^2}$ and $\widetilde{\Theta}=\frac{\mu s}{2}
$.

As in the first case, we use (\ref{arclength}) and the half-angle substitution $z=\tan\frac{u}{2}$ to calculate:

$$
s=
\int_{\tan(\phi_0/2)}^{\tan(\phi/2)}
\frac{2\,dz}{2z-\xi_0(1+z^2)}.
$$
For $\xi_0\neq0$, the integration yields
\begin{equation}\label{eq: s1}
    s=
\frac{1}{\mu}
\left[
\log\left|
\frac{\mu+\xi_0\tan(u/2)-1}
{\mu-\xi_0\tan(u/2)+1}
\right|
\right]_{u=\phi_0}^{u=\phi}.
\end{equation}
For $\xi_0=0$, the integral is computed separately:
\begin{equation}\label{eq: s2}
    s=
\int_{\phi_0}^{\phi}\frac{du}{\sin u}
=
\left[
\log\left|\tan\frac{u}{2}\right|
\right]_{u=\phi_0}^{u=\phi}.
\end{equation}
We rewrite both \eqref{eq: s1} and \eqref{eq: s2} using
the hyperbolic trigonometric identity $
\tanh(x-y)
=
\frac{\tanh x-\tanh y}
{1-\tanh x\tanh y},
$ for all $x,y\in \R$, and obtain in both cases:
\begin{equation}\label{eq: Jh0}
    \tanh\widetilde{\Theta}
=
\frac{
\mu\sin\left(\frac{\phi-\phi_0}{2}\right)
}{
\sin\left(\frac{\phi+\phi_0}{2}\right)
-
\xi_0\cos\left(\frac{\phi-\phi_0}{2}\right)
}.
\end{equation}
We rewrite the latter identity as
\begin{equation}\label{eq: Jh1}
    \sin\frac{\phi-\phi_0}{2}
=
\frac{\tanh\widetilde{\Theta}}{\mu}
\left(
\sin\frac{\phi+\phi_0}{2}
-
\xi_0\cos\frac{\phi-\phi_0}{2}
\right).
\end{equation}
After making use of \eqref{eq: Jh0}, the following identity holds:
\begin{align}\label{eq: Jh2}
(\sin\phi-\xi_0)(\sin\phi_0-\xi_0)
&=
\left(
\sin\frac{\phi+\phi_0}{2}
-
\xi_0\cos\frac{\phi-\phi_0}{2}
\right)^2
(1-\tanh^2\widetilde{\Theta})
=
\frac{
\left(
\sin\frac{\phi+\phi_0}{2}
-
\xi_0\cos\frac{\phi-\phi_0}{2}
\right)^2
}{
\cosh^2\widetilde{\Theta}
}.
\end{align}
Moreover, since
$
s=\frac{2\widetilde{\Theta}}{\mu}$ and
$\xi_0^2-1=-\mu^2,$
we have
\begin{align}\label{eq: Jh3}
&2\sin\frac{\phi-\phi_0}{2}
+
s\left(
\xi_0\cos\frac{\phi-\phi_0}{2}
-
\sin\frac{\phi+\phi_0}{2}
\right)
=
\frac{2}{\mu}
\left(
\sin\frac{\phi+\phi_0}{2}
-
\xi_0\cos\frac{\phi-\phi_0}{2}
\right)
(\tanh\widetilde{\Theta}-\widetilde{\Theta}).
\end{align}
Substituting \eqref{eq: Jh1}, \eqref{eq: Jh2} and \eqref{eq: Jh3} into the expression for $J$ given in Lemma \ref{lem: J}, yields
$$
J
=
\frac{4}{\mu^4}
\tanh\widetilde{\Theta}\,\cosh^2\widetilde{\Theta}
(\widetilde{\Theta}-\tanh\widetilde{\Theta}).
$$
Finally, using
$
\widetilde{\Theta}-\tanh\widetilde{\Theta}
=
\frac{\widetilde{\Theta}\cosh\widetilde{\Theta}-\sinh\widetilde{\Theta}}{\cosh\widetilde{\Theta}}
$
and 
 $\mu^4=(1-\xi_0^2)^2$, the latter equation becomes
\begin{equation}
J=\frac{4}{(1-\xi_0^2)^2}
\sinh\widetilde{\Theta}
\bigl(
\widetilde{\Theta}\cosh\widetilde{\Theta}-\sinh\widetilde{\Theta}
\bigr).\label{detJ2}
\end{equation}
Since $\mu>0$ and $s>0$, we have $\widetilde{\Theta}>0$, and consequently
$\sinh\widetilde{\Theta}>0$.
To determine the sign of the second factor, we define
$
G(\widetilde{\Theta})
=
\widetilde{\Theta}\cosh\widetilde{\Theta}-\sinh\widetilde{\Theta}, \;{\Theta}>0.
$
Then,
$G(0)=0$
and
$
G'(\widetilde{\Theta})
=
\widetilde{\Theta}\sinh\widetilde{\Theta}>0
,$ 
for all $\widetilde{\Theta}>0.$
Thus,
$G(\widetilde{\Theta})>0$ for all
$\widetilde{\Theta}>0.$
Consequently,
$
 J>0,
\text{ for all } s>0.
$
Hence, the exponential map has no critical point at any positive time and
there is no finite conjugate time, i.e. 
$T_{\mathrm{conj}}(\xi_0,\phi_0)
=
+\infty,
\text{ for }
-1<\xi_0<1.$

\noindent \textit{Case III:} $\xi_0=-1$. The expression (\ref{detJ2}) obtained in Case II appears singular at $\xi_0=-1$.
This singularity is removable. Indeed, the Hamiltonian vector field is
smooth, and its flow depends smoothly on the initial covector. Consequently,
$\exp_\e$, its differential, and the determinant $J$ depend
continuously on $\xi_0$. Because we assumed $\sin\phi_0-\xi_0>0$, at $\xi_0=-1,$ we have $\sin\phi_0+1>0$. This, also holds for $\xi_0>-1$ sufficiently close to $-1$, which means that (\ref{detJ2}) can be used by taking the limit at $\xi_0\to-1^+$:
$$
J(-1,\phi_0,s)
=
\lim_{\xi_0\to-1^+}J(\xi_0,\phi_0,s).
$$
The Taylor expansions
$
\sinh\widetilde{\Theta}=\widetilde{\Theta}+O(\widetilde{\Theta}^3)$ and  
$\widetilde{\Theta}\cosh\widetilde{\Theta}-\sinh\widetilde{\Theta}
=
\frac{\widetilde{\Theta}^3}{3}+O(\widetilde{\Theta}^5),
$ for $\widetilde{\Theta}\to0$, together with the setting for Case II given by $\mu=\sqrt{1-\xi_0^2}$ and $
\widetilde{\Theta}=\frac{\mu s}{2}$ (since $\mu\to 0^+$, as $\xi_0\to -1^+$), 
imply
$$
J(-1,\phi_0,s)
=
\lim_{\mu\to0^+}
\frac{4}{\mu^4}
\sinh\widetilde{\Theta}
\bigl(
\Theta\cosh\widetilde{\Theta}-\sinh\widetilde{\Theta}
\bigr)
=
\frac{s^4}{12}.
$$
Thus, we conclude that
$
J(-1,\phi_0,s)>0
$ for all $s>0$. In other words,
$
T_{\mathrm{conj}}(-1,\phi_0)=+\infty.
$
To complete the proof, we consider the case $\sin\phi_0-\xi_0<0$. We recall the induced reflection \eqref{e1} on the costate $C_\e$, given by $\varepsilon_1(\xi_0,\phi_0,s)=(-\xi_0,-\phi_0,s)$. We have the following composition of tangent maps, induced by the symmetry $(\varepsilon_1,j)$:
$$\left(D\,\exp_\e\right)_{(-\xi_0,-\phi_0,s)}\circ D\,{\varepsilon_1}_{(\xi_0,\phi_0,s)}=D\,j_{\exp_\e(\xi_0,\phi_0,s)}\circ \left(D\,\exp_\e \right)_{(\xi_0,\phi_0,s)}.$$
Since $D\varepsilon_1$ and $Dj$ are invertible, it follows that
$\mathrm{rank}
\left(D\exp_\e\rvert_{(\xi_0,\phi_0,s)}\right)=\mathrm{rank}
\left(D\exp_\e\rvert_{(-\xi_0,-\phi_0,s)}\right).$
As a consequence, the conjugate times are preserved by this symmetry. In particular, if $\xi_0>1$, then $-\xi_0<-1$, and the preceding computation gives $T_\mathrm{conj}(\xi_0,\phi_0)=T_\mathrm{conj}(-\xi_0,-\phi_0)=\frac{2\pi}{\sqrt{\xi_0^2-1}}$.
Similarly, the case $\xi_0=1$, follows from the already established case $\xi_0=-1$, leading to an infinite conjugate time. For $\xi_0\in (-1,1)$, the same symmetry $(\varepsilon_1,j)$ maps the sector $\sin\phi_0-\xi_0<0$ to the sector $\sin\phi_0-\xi_0>0$ (already analyzed) and no finite conjugate time occurs.

Finally, for $\sin\phi_0-\xi_0=0$ the conclusion follows by direct differentiation of the explicit geodesics in cases $(1)-(2)$ of Theorem \ref{thm: Geodesics of AA}. Thus, no finite conjugate time occurs in this case. After examining all cases, the proof is now complete.
\end{proof}

In order to prove that inequality (\ref{MAXbound}) holds as an equality, we will apply Theorem \ref{Hadamard} in the following setting.

We let \begin{equation}\label{eq: X,Y}
    \begin{aligned}
X_0&= \left\{(\xi_0,\phi_0,s): \, |\xi_0|\leq 1, \phi_0\in \mathbb{S}^1, s\in (0,+\infty)  \right\}\\
X&=\left\{(\xi_0,\phi_0,s)\, :|\xi_0|>1,\phi_0\in \mathbb{S}^1,s \in \left(0 , T_{\xi_0}\right) \right\}\cup X_0\\
Y&=\Aa\setminus \overline{\text{Max}},\end{aligned}
\end{equation}
where $\overline{\mathrm{Max}}=\left\{(1,0,a) \, : \, a\in \R \right\}$ denotes the closure of $\text{Max}$ and $T_{\xi_0}$ is given by (\ref{period}) . We consider the restricted exponential map $F$ given by  
\begin{equation}\label{eq: F}
    F:=\mathrm{exp_e}|_X : X\to Y.
\end{equation}
The following results will turn useful later.
\begin{lem}\label{Lem: WD}
    Let $X, Y$ be the spaces and $F:X\to Y$ be the map given by \eqref{eq: X,Y} and \eqref{eq: F}, respectively. The map $F:X\to Y$ is well-defined.
\end{lem}
\begin{proof}
We are going to prove that $F(X)\subset Y$.  We assume $\sin\phi_0-\xi_0\neq0$ and use the definition of exponential map \eqref{Expe}. For all $\left( \xi_0,\phi_0,s\right)\in X\setminus X_0$ the following holds
\begin{equation}F(\xi_0,\phi_0,s)\in \overline{\text{Max}}\iff \sin\phi(s)=\sin\phi_0, \, \cos\phi(s)=\cos\phi_0\iff \phi(s)=\phi_0+2k\pi, \quad k\in \mathbb{Z}.\label{imF}\end{equation}
The last expression leads to a contradiction, since $s\in (0,T_{\xi_0})$ or equivalently $\phi$ does not complete a full rotation. If $\left( \xi_0,\phi_0,s\right)\in X_0$, similarly by \eqref{imF}, it must be $\phi(s)=\phi_0+2k\pi, \quad k\in \mathbb{Z}$. Recalling the proof of Prop \ref{lem: ang period}, this case yields $|\phi(s)-\phi_0|<2\pi$, for all $s>0$. Hence, $F(\xi_0,\phi_0,s)\notin \overline{\text{Max}}$.
If $\sin\phi_0-\xi_0=0$, we are dealing with the cases $(1)-(2)$ of Theorem \ref{thm: Geodesics of AA}. We have a contradiction too, because $(\lambda(s),t(s))\neq (1,0)$ for all $s>0$.

\end{proof} 
\begin{lem}\label{PropernessOfF}
   Let $X, Y$ be the spaces and $F:X\to Y$ be the map given by \eqref{eq: X,Y} and \eqref{eq: F}, respectively. The map $F:X\to Y$ is proper.
\end{lem}

\begin{proof}
    Let $K \subset Y$ be compact. We consider sequences $(\xi_0^n,\phi_0^n)_{n\in\mathbb N}\subset \R\times \mathbb{S}^1$ and $(s_n)_{n\in\mathbb N}\subset\R_+$ such that $(\xi_0^n,\phi_0^n,s_n)_{n\in\mathbb N} \in F^{-1}(K)$ and denote 
    \begin{align}\label{eq: cpt cond}
        \gamma(s_n;\xi_0^n,\phi_0^n)=F(\xi_0^n,\phi_0^n,s_n)  \in K,\quad\forall n\in\mathbb{N},
    \end{align}
    where $\gamma(\cdot\,;\xi_0^n,\phi_0^n)$ corresponds to the normal extremal with initial covector $(\xi_0^n,\phi_0^n)$.
    Due to compactness of $K$ we extract a convergent subsequence $\gamma(s_{n_k};\xi_0^{n_k},\phi_0^{n_k})$ in $Y$ and indicate the limit with $q \in K$. 
    
    We aim at finding bounds for $\phi_0^{n_k}$, $s_{n_k}$ and $\xi_0^{n_k}$ for all $k\in\mathbb N$. We observe that
$(\phi_0^{n_k})_{k\in\mathbb{N}} \in \mathbb{S}^1$ is clearly bounded. 
    
    We notice that $(s_n)_{n\in\mathbb{N}}$ is also bounded. Indeed, we have the following case distinction. \\
    \noindent{\textit{Case I:}} If $|\xi_0|>1$, then $0< s<T_{\xi_0}.$ Here it holds that 
    \begin{align*}
        \gamma (s;\xi_0,\phi_0) \to \left(1,0,a \right) \in \text{Max},\quad \text{as } s\to T_{\xi_0}.
    \end{align*}
    If we had a subsequence $(s_{n_k})_{k\in\mathbb N}$ of $(s_n)_{n\in\mathbb N}$ with $s_{n_k} \to T_{\xi_0^{n_k}}$, the corresponding endpoints $\gamma(s_{n_k};\xi_0^{n_k},\phi_0^{n_k})$ would converge to Maxwell points, contradicting that $\gamma(s_{n_k};\xi_0^{n_k},\phi_0^{n_k})\to q \in K \subset Y$. 
    
    Now, suppose there is a subsequence $(s_{n_k})_{k\in\mathbb N}$ with $s_{n_k}\to 0.$ Here it holds instead that \begin{align*}
        \gamma (s;\xi_0,\phi_0) \to \e \in \overline{\text{Max}} , \quad \text{as } s\to 0.
    \end{align*}
Therefore, for sufficiently large $k\in\mathbb N$ there exists $\delta >0$, independent of $k$, such that $ s_{n_k} \ge \delta.$

 \noindent{\textit{Case II:}}  If $|\xi_0^{n_k}|\leq1$, the Maxwell time is $T_{\text{Max}}=+\infty$. As $s\to + \infty$, at least one component of $\gamma(\cdot\,;\xi_0^{n_k},\phi_0^{n_k})$, given by $\lambda (\cdot),|t(\cdot)|,|a(\cdot)|$, becomes unbounded. In this way $\gamma (s_{n_k};\xi_0^{n_k},\phi_0^{n_k})$ would escape compact sets for large $k$, contradicting \eqref{eq: cpt cond}. 

 Now, we argue by contradiction to bound $(\xi_0^{n})_{n\in\mathbb{N}}$. Let $(\xi_0^{n_k})_{k\in\mathbb{N}}$ be a subsequence such that $|\xi_0^{n_k}| \to +\infty$. Then, for bounded $s$, at least one of $\lambda (s),|t(s)|,|a(s)|$ becomes again unbounded and it contradicts \eqref{eq: cpt cond}. 

To sum up, $(\xi_0^n,\phi_0^n,s_n)_{n\in\mathbb{N}}$ is contained in a closed ball of $\mathbb{R}^3,$ and it has a convergent subsequence $(\xi_0^{n_k},\phi_0^{n_k},s_{n_k})_{
k\in\mathbb{N}}$ with limit $(\xi_0^*,\phi_0^*,s^*).$
Continuity implies that for $
        (\xi_0^{n_k},\phi_0^{n_k},s_{n_k}) \to (\xi_0^*,\phi_0^*,s^*)$ we have
        $$F(\xi_0^{n_k},\phi_0^{n_k},s_{n_k}) = \gamma (s_{n_k};\xi_0^{n_k},\phi_0^{n_k})\to F(\xi_0^*,\phi_0^*,s^*)= q.$$
    We will show that $(\xi_0^*,\phi_0^*,s^*) \in F^{-1}(K),$ i.e. $s^* \in (0,T_{\text{Max}}(\xi_0^*,\phi_0^*))$.
    
    If $s^* =0$, we have $F(\xi_0^*,\phi_0^*,s^*)=(1,0,0)$, which does not lie in the image of $F$ (since $s>0$ in $X$), thus contradicting that $\gamma (s_{n_k};\xi_0^{n_k},\phi_0^{n_k}) \to q.$
    
    If $s^*=T_{\xi_0^*}$ (for $\xi_0^*<-1$ without loss of generality), then we have $
        F(\xi_0^*,\phi_0^*,s^*) \in \text{Max} \subset \overline{\text{Max}}$, thus
    contradicting that $q\in K \subset Y= \Aa \setminus \overline{\text{Max}} $.
\end{proof}

\begin{thm}\label{Thm: cut-time}
    The cut time of $\gamma:\R_+\to\Aa$ with initial data $\gamma(0)=\e$ and $(\xi_0,\phi_0)\in\R\times\mathbb{S}^1$ is 
    \begin{align}\label{eq: cut time}
        T_{\mathrm{cut}} &=\begin{cases}
            \frac{2\pi}{\sqrt{\xi_0^2-1}}, & \text{if }|\xi_0|>1 \\
            +\infty, & \text{if }-1\leq \xi_0 \leq 1.
        \end{cases}
    \end{align}
\end{thm}
\begin{proof}
By choosing $X,Y$ and $F$ as in $\eqref{eq: X,Y}$ and $\eqref{eq: F}$ we have that Lemma \ref{Lem: WD} assures us that $F:X\to Y$ is a well-defined map. We are going to apply Theorem \ref{Hadamard} after having verified properties (1)-(4):
\begin{itemize}
    \item Since $X,Y$ are open sets of three-dimensional manifolds, we have $\dim X=\dim Y=3$.
    \item It is trivial by their definitions that $X$ and $Y$ are connected.   
    \item $F:X\to Y$ being a local diffeomorphism follows after applying the Inverse function theorem to Proposition \ref{conj}. Indeed,  we have that  
    $ J\neq 0,$ for all $ s\in \left(0,T_{\xi_0}\right)$,
where $J=\det\left( D\,\exp_\e\right)$ is calculated on $(\xi_0,\phi_0,s)\in X$ and with respect to the bases $\left\{ \partial_{\xi_0},\partial_{\phi_0},\partial_s\right\}\rvert_{(\xi_0,\phi_0,s)}$ of ${\rm T}_{(\xi_0,\phi_0,s)}C_\e\times \R_+$ and $\left\{U,V,W\right\}\rvert_{\gamma(s)}$ of ${\rm T}_{\gamma(s)}\Aa$.
\item $F$ is proper by Lemma \ref{PropernessOfF}.
\end{itemize}

Therefore $F$ is a finite-sheeted covering map. By definition of $X$ and $Y$, the following homeomorphisms hold:
\begin{align*}
    X&\cong \R \times\mathbb{S}^1\times (0,+\infty),\\
    Y&\cong\left(\R^2\setminus\left\{(0,0) \right\}\right)\times\R.
\end{align*}
A basic fact on fundamental groups (see \cite[Proposition 1.12]{Hatcher}) yields
$$\pi_1\left(X\right)\cong\pi_1(Y)\cong \mathbb{Z}.$$
We denote by $F_*$ the induced homomorphism $F_* : \mathbb{Z}\longrightarrow \mathbb{Z}$. Following the preliminary observation from Section \ref{Preliminaries}, we have that the number of sheets is equal to the index
$$[\mathbb{Z} : F_*\mathbb{Z}]=[\mathbb{Z} : \mathbb{Z}]=1.$$
Since $X$ is connected, the cardinality of the fibers of $F$ is constant. The latter fact combined with the definition of number of sheets yields  $|F^{-1}(y)|=[\mathbb{Z} : F_*\mathbb{Z}]=1$, for all $y\in Y$. Hence, $F$ is bijective as a covering map with one sheet. Combining with the fact that $F$ is a local diffeomorphism we obtain that $F : X\to Y$ is a global diffeomorphism. Therefore the cut time is the first Maxwell time, which is given by \eqref{period}.
\end{proof}

\section{Surfaces generated by geodesics}\label{sec-ccflask}
\subsection{The CC-unit sphere}
Let $\e\in\Aa$ and $r>0$, the \textit{CC sphere of radius $r$ centered at $\e$} is the set
\begin{equation*}
    \mathcal{S}_\Aa(\e,r):=\{p\in\Aa:d_\Aa(\e,p)=r\}.
\end{equation*}
The main result of this section is the following
\begin{thm}\label{thm: CCSphere}
The CC-unit sphere $\mathcal{S}_\Aa(\e,1)$ is parametrized by initial data $(\eta_0,\tau_0,\xi_0)$ which satisfies $\eta_0^2+(\xi_0+\tau_0)^2=1$. We set $\phi_0\in\mathbb{S}^1$ such that $\cos\phi_0=\eta_0$, $\sin\phi_0=\xi_0+\tau_0$ and define:
\begin{align*}
    \Xi&=\frac{1 +\beta_1 \tanh \left(\frac{1}{2} \beta_1 (c_1+1)\right)}{\xi_0},\\
    \Psi&=\frac{1 -\beta_2 \tan \left(\frac{1}{2} \beta_2 (c_2+1)\right)}{\xi_0},\\
    \Omega&=-\frac{3+c_3}{1+c_3},
\end{align*}
with $\beta_1=\sqrt{1-\xi_0^2}$ and $\beta_2=\sqrt{\xi_0^2-1} .$
The sphere is given by the surface patches $(\lambda, t,a)$ in the following case distinction.\\
If $\tau_0\neq 0$,\\
\noindent{\it  1. if} $|\xi_0| < 1$: 
then
\begin{align*}
\lambda &= \frac{2\Xi   - \xi_0(1+\Xi^2)}{(1+\Xi^2)(\sin\phi_0 - \xi_0)} , \\
t &= \frac{(1+\Xi^2)\cos \phi_0-(1-\Xi^2)}{(1+\Xi^2)(\sin \phi_0 -\xi_0)},\\
a &= 2 \arctan \Xi+ \xi_0 - \phi_0, 
\end{align*}
for $\xi_0\neq 0$, and
\begin{align*}
\lambda &=\frac{\pm2ec_1'   - \xi_0(1+(ec_1')^2)}{(1+(ec_1')^2)(\sin\phi_0 - \xi_0)} , \\
t &=  \frac{(1+(ec_1')^2)\cos \phi_0-(1-(ec_1')^2)}{(1+(ec_1)^2)(\sin \phi_0 -\xi_0)},\\
    a &= 2 \arctan (\pm ec_1')+ \xi_0 - \phi_0,
\end{align*}
for $\xi_0=0$,

\noindent{\it 2. if} $\xi_0 < -1$:
then
\begin{align*}
\lambda &= \frac{2\Psi   - \xi_0(1+\Psi^2)}{(1+\Psi^2)(\sin\phi_0 - \xi_0)}, \\
t &=  \frac{(1+\Psi^2)\cos \phi_0-(1-\Psi^2)}{(1+\Psi^2)(\sin \phi_0 -\xi_0)},\\
a &= 2 \arctan \Psi+ \xi_0 - \phi_0,
\end{align*}
with $1< \frac{2\pi}{\sqrt{\xi_0^2-1}},$

\noindent{\it 3. if} $\xi_0 = -1$:
then
\begin{align*}
\lambda &= \frac{2\Omega   +(1+\Omega^2)}{(1+\Omega^2)(\sin\phi_0 +1)} , \\
t &=  \frac{(1+\Omega^2)\cos \phi_0-(1-\Omega^2)}{(1+\Omega^2)(\sin \phi_0 +1)},\\
a &= 2 \arctan \Omega - 1 - \phi_0.
\end{align*}

The constants $c_1,c_1', c_2, c_3 \in \mathbb{R}$ are uniquely determined by $\phi_0$. The case $\xi_0\geq1$ is recovered by the action of the reflection $j:\Aa\to\Aa$.

If $\tau_0=0$, we have the following cases 
\begin{align*}
    \gamma (1)&=(1,\xi_0,\xi_0), \quad \text{if } \eta_0=0,\\
    \gamma (1)&= \left(e^{\eta_0},\frac{\xi_0}{\eta_0}(e^{\eta_0}-1),\xi_0 \right), \quad \text{if } \eta_0\neq 0.
\end{align*}
The poles of $\mathcal{S}_\Aa(\e,1)$ are the points $P_N=\left(1,0,2\pi-\sqrt{1+4\pi^2}\right)$ and $P_S=\left(1,0,-2\pi+\sqrt{1+4\pi^2}\right)$.

\end{thm}

\begin{proof}
The Carnot-Carath\'eodory sphere of radius $1$ centered at the identity $\e = (1,0,0)$ is defined as the locus of endpoints of all geodesics evaluated at $s=1$.

By Theorem \ref{thm: Geodesics of AA}, for $\tau_0 \neq 0$, the arc-length parametrized geodesics $\gamma(s) = (\lambda(s), t(s),a(s))$ are given by:
\begin{align*}
\lambda(s) &= \frac{\sin\phi(s)-\xi_0}{\sin\phi_0-\xi_0} \\
t(s) &= \frac{\cos\phi_0-\cos\phi(s)}{\sin\phi_0-\xi_0}\\
a(s) &= \phi(s) - \phi_0 + \xi_0 s
\end{align*}
where we have applied the initial covector relation $\tau_0 = \sin\phi_0 - \xi_0$.

To construct the unitary sphere, we evaluate these curves at $s=1$, yielding the spatial coordinates $(\lambda, t,a) = (\lambda(1), t(1),a(1))$. Substituting the explicit closed-form solutions for $\phi(1)$ derived in Theorem \ref{thm: Geodesics of AA} for the three regimes $|\xi_0| < 1$, $\xi_0 < -1$, and $\xi_0 = -1$ generates a family of surface patches.

The latter collection has to be compared with the study of cut times from Theorem \ref{Thm: cut-time} to give the surface patch parametrization presented. 
In detail, from $s=1\leq T_{\mathrm{cut}}$ combined with (\ref{eq: cut time}) for $|\xi_0|>1$, we find $|\xi_0|\leq \sqrt{1+4\pi^2}$. Following Lemma \ref{lem: ang period}, the poles of the unit Carnot-Carath\'eodory sphere are the points $\gamma(T_{\xi_0})$ and its reflection by $j$, where $\xi_0=\sqrt{1+4\pi^2}$. A straightforward computation yields 
\begin{align*}
    P_N=\left(1,0,2\pi-\sqrt{1+4\pi^2}\right), \qquad
    P_S=j(P_N)=\left(1,0,-2\pi+\sqrt{1+4\pi^2}\right).
\end{align*}
\end{proof} 
\begin{rem}
    The poles of the CC-unit sphere $\mathcal{S}_\Aa(\e,1)$ are the intersection of $\mathcal{S}_\Aa(\e,1)$ with the $a$-axis translated by $\e$.
\end{rem}

Keeping in mind Theorem \ref{Thm: cut-time}, one may derive an analogous statement to Theorem \ref{thm: CCSphere} for general CC-spheres of radius $r>0$.
In order to visualize the profile of a CC-sphere at different radii we plot $\mathcal{S}_\Aa(\e,1)$ and $\mathcal{S}_\Aa(\e,5)$.

\begin{figure}[H]
\centering
\begin{subfigure}{0.65\textwidth}
\includegraphics[width=1\textwidth]{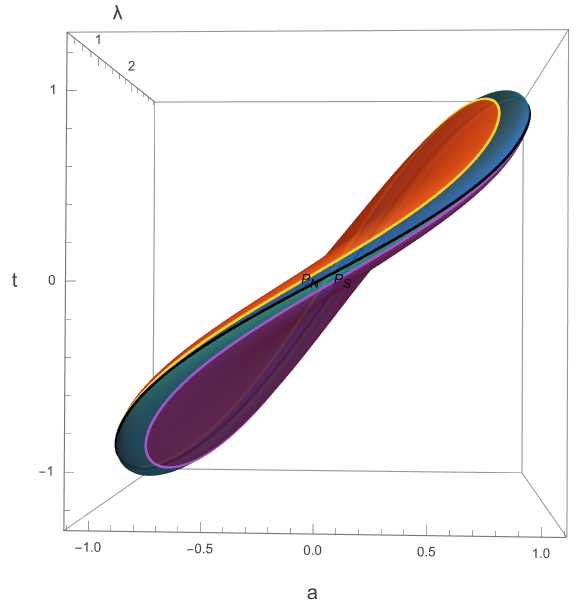}
\end{subfigure}
\begin{subfigure}{0.495\textwidth}
    \includegraphics[width=1\textwidth]{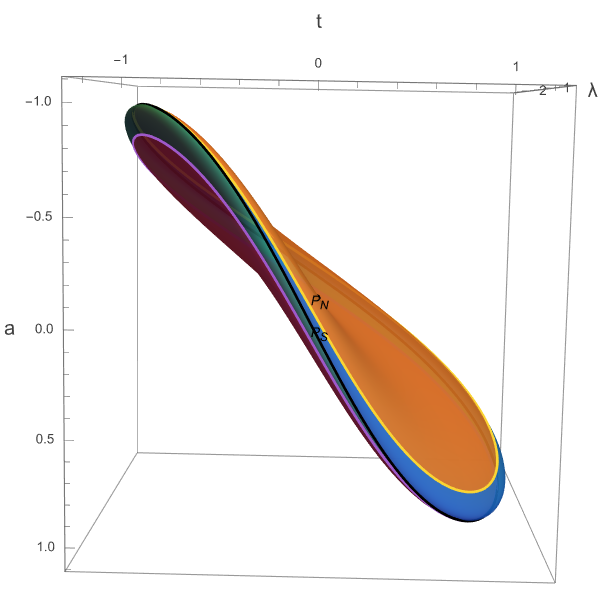}
\end{subfigure}
\begin{subfigure}{0.495\textwidth}
    \includegraphics[width=1\textwidth]{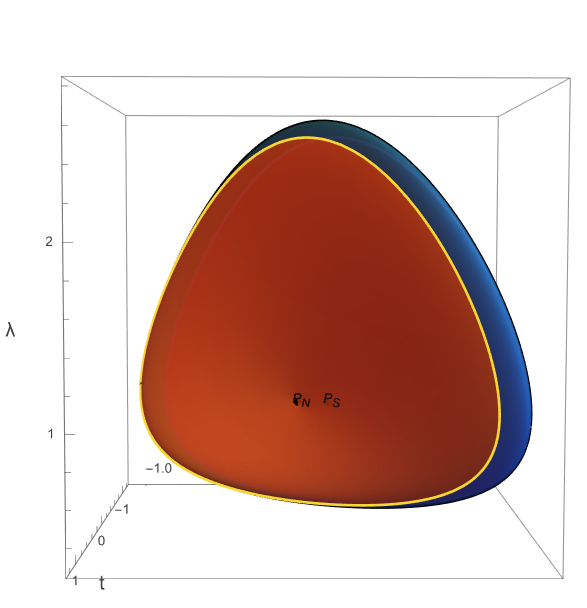}
\end{subfigure}
\caption{Plot of $\mathcal{S}_{CC}(\e,1)$. \\ The blue region corresponds to the case $|\xi_0|<1$ with $\tau_0>0$, \\
the teal region corresponds to the case $|\xi_0|<1$ with $\tau_0<0$,\\
the orange region corresponds to the case $\xi_0<-1$,\\ the magenta region corresponds to the case $\xi_0>1$, \\ the yellow and purple curves correspond to the critical cases for $\xi_0=\pm 1$ \\
and the black curve corresponds to the case $\tau_0=0$.}\label{pic: CC-sphere r=1}
\end{figure}

\begin{figure}[H]
\centering
\begin{subfigure}{0.65\textwidth}
\includegraphics[width=1\textwidth]{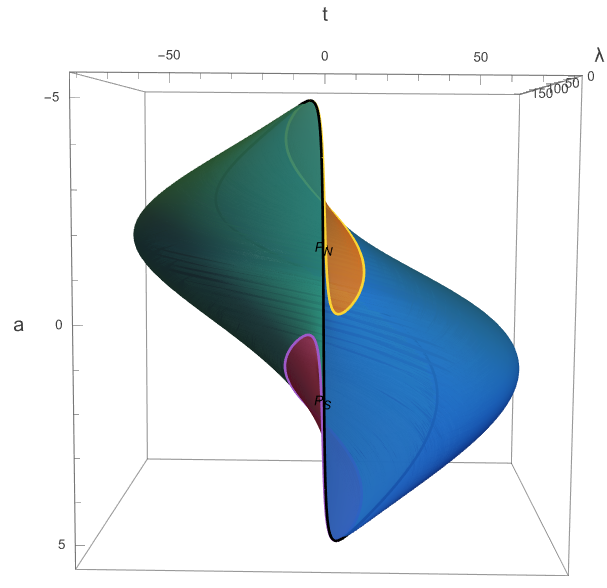}
\end{subfigure}
\begin{subfigure}{0.495\textwidth}
    \includegraphics[width=1\textwidth]{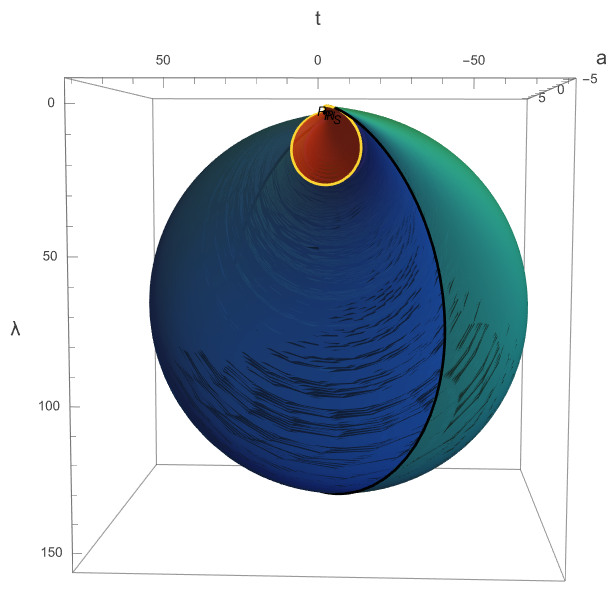}
\end{subfigure}
\begin{subfigure}{0.495\textwidth}
    \includegraphics[width=1\textwidth]{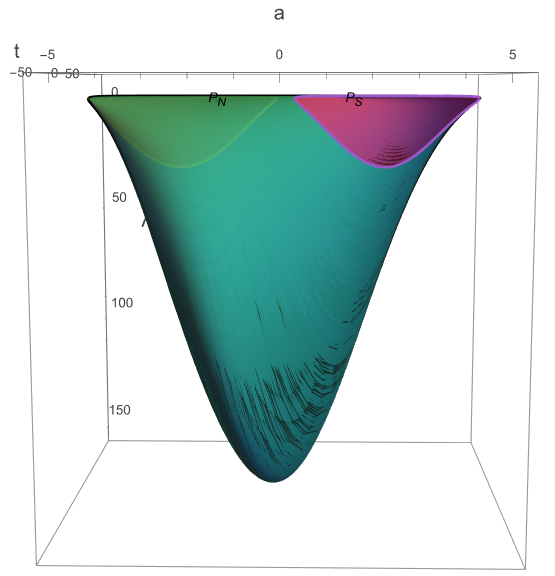}
\end{subfigure}
\caption{Plot of $\mathcal{S}_{CC}(\e,5)$. \\The blue region corresponds to the case $|\xi_0|<1$ with $\tau_0>0$, \\
the teal region corresponds to the case $|\xi_0|<1$ with $\tau_0<0$,\\
the orange region corresponds to the case $\xi_0<-1$,\\ the magenta region corresponds to the case $\xi_0>1$, \\ the yellow and purple curves correspond to the critical cases for $\xi_0=\pm 1$ \\
and the black curve corresponds to the case $\tau_0=0$.}\label{pic: CC-sphere r=5}
\end{figure}

\subsection{The flask}
In this section, we present a construction in $\Aa$ of the analogue of the Pansu sphere (bubble set) in the first Heisenberg group (see \cite{Pansu1} or \cite[Section 2.3]{CDPT}).

We need first to analyze the sub-Riemannian geodesics of $\Aa$ with respect to the canonical projection $\pi : \Aa \longrightarrow \bH_{\C}^1$.
In detail, let $D\pi : \rm{T}\Aa\longrightarrow \rm{T}\,\bH_{\C}^1$ be the tangent map which sends the horizontal orthonormal distribution $\mathcal{H}_\Aa=\text{span}\{U,V\}$ to the orthonormal frame
$$D\pi(U)=\lambda\partial_\lambda:=f_1, \quad D\pi(V)=\lambda\partial_t:=f_2$$
of the hyperbolic plane $(\bH_{\C}^1,g_{\text{hyp}})$. Let $\widetilde\gamma(s)=(\lambda(s),t(s))$ be a unit-speed curve in
$\bH_{\C}^1$.  With respect to the hyperbolic orthonormal frame $\{f_1,f_2\}$, we write
\begin{equation*}\label{eq:tangent-angle}
 \dot{\widetilde\gamma}(s)
 =\cos\Phi(s)f_1+\sin\Phi(s)f_2,
\end{equation*}
where $\Phi(s)$ is an angle function.
The coframe dual to $f_1,f_2$ is
$
 \omega_1=\frac{d\lambda}{\lambda},
 \omega_2=\frac{dt}{\lambda},
$
and the structure equations 
$$d\omega_1=\eta\wedge \omega_2, \quad d\omega_2=-\eta\wedge\omega_1$$
are satisfied.
Therefore, the connection form is $\eta=-\omega_2$, and the signed geodesic
curvature of $\widetilde\gamma$ is defined via \cite[Section 4.4.2]{ABB}:
\begin{equation}\label{kg}
 k_g^{\widetilde\gamma}(s)
 :=\dot\Phi(s)+\eta(\dot{\widetilde\gamma}(s))
 =\dot\Phi(s)-\sin\Phi(s).
\end{equation}

\begin{prop}\label{mom}
The momentum $\xi_0$ represents the signed curvature of the projection of a sub-Riemannian geodesic $\gamma : [0,T] \to \Aa$ (such that $\gamma(0)=\e$ and $\xi(0)=(\eta_0,\tau_0,\xi_0)$) to a curve $\widetilde\gamma$ on the hyperbolic plane, i.e.
\begin{equation}
     k_g^{\widetilde\gamma}=-\xi_0.
\end{equation}
\end{prop}

\begin{proof}
Let $\widetilde\gamma:=\pi\circ \gamma$. Because of our analysis in Section \ref{sec: PMP}, we have
\begin{equation}\dot\gamma=\cos\phi \,U+\sin\phi \,V\Longrightarrow \dot{\widetilde{\gamma}}=D\pi_{\gamma}(\dot\gamma)=\cos\phi(s)f_1+\sin\phi(s)f_2.\label{Dpg}\end{equation}
By applying formula (\ref{kg}) combined with \eqref{eq: angle ode}, we have $ k_g^{\widetilde\gamma}=\dot\phi-\sin\phi=-\xi_0.$
\end{proof}

Theorem \ref{thm: Geodesics of AA}, combined with Remark \ref{rem: proj} and Proposition \ref{mom}, leads to the following (for the precise definitions of horocycles and hypercycles, i.e. equidistant curves, we refer the reader to \cite[Chapter 16]{Coxeter}).

\begin{cor}\label{corcurv}
Let $\gamma:I\to\Aa$ be a
sub-Riemannian geodesic, and let
$\widetilde{\gamma}=\pi\circ\gamma:I\to\mathbb \bH_{\C}^1$
be its projection onto the hyperbolic plane. Then,
$\widetilde{\gamma}$ represents one of the following cases, depending on the range of $\xi_0$:

\begin{enumerate}

\item If $\xi_0=0$, then $\widetilde{\gamma}$ is a hyperbolic
geodesic (i.e. a perpendicular line to the axis $\lambda=0$ or a semi-circle centered at the $t-$axis).

\item If $0<|\xi_0|<1$, then $\widetilde{\gamma}$ is a hypercycle.
If $r$ denotes its hyperbolic distance from the corresponding
geodesic, then $|\xi_0|=\tanh r$ and
$r=\operatorname{arctanh}|\xi_0|$.

\item If $|\xi_0|=1$, then $\widetilde{\gamma}$ is a horocycle.

\item If $|\xi_0|>1$, then $\widetilde{\gamma}$ is a hyperbolic
circle. Its hyperbolic radius $R$ satisfies
$|\xi_0|=\coth R$ and
$R=\operatorname{arccoth}|\xi_0|$.
\end{enumerate}

The sign of $\xi_0$ determines the orientation of the projected
curve, since $k_g^{\widetilde\gamma}=-\xi_0$.  
\end{cor}

Before proceeding with the construction of the flask surface we set up basic notation and terminology.
The hyperbolic distance $d_{\mathrm{hyp}}$ between two points $p_1=(\lambda_1, t_1)$ and $p_2=(\lambda_2, t_2)$ in the right hyperbolic plane $\bH^1_\C$ is given by the relation
\begin{equation*}\label{hyp-distance}
    \cosh \left(d_{\mathrm{hyp}}(p_1,p_2)\right) = 1 + \frac{(t_1 - t_2)^2 + (\lambda_1 - \lambda_2)^2}{2\lambda_1 \lambda_2}.
\end{equation*}

Fix a point $p=(\lambda_0,t_0)\in\bH_{\C}^1$ and a (hyperbolic) radius $R>0$.  The hyperbolic
circle $C_p(R)$ is defined by
$
 d_{\mathrm{hyp}}\bigl((\lambda,t),(\lambda_0,t_0)\bigr)=R.$
Using the hyperbolic distance formula, its equation becomes
\begin{equation*}\label{hypcirclee}
 (\lambda-\lambda_0\cosh R)^2+(t-t_0)^2
 =(\lambda_0\sinh R)^2.
\end{equation*}
Hence $C_p(R)$ is the Euclidean circle with center
$(\lambda_0\cosh R,t_0)$ and radius $\lambda_0\sinh R$.  A convenient
parametrization is given by
\begin{equation}\label{hypcircle1}
 \begin{aligned}
 \lambda(\psi)&=\lambda_0
   (\cosh R+\sinh R\cos\psi),\\
 t(\psi)&=t_0+\lambda_0\sinh R\sin\psi,
 \end{aligned} 
\end{equation}
where $\psi\in(-\pi,\pi).$

Now, consider a curve $\gamma : [0,T]\to \Aa$ to be an extremal trajectory starting at $\gamma(0)=\e$ and having endpoint $\gamma(T)$ on the $a$-fiber over $(1,0)\in\mathbf{H}^1_\C$, such that $T\leq T_{\mathrm{cut}}$. By the horizontality condition and Stokes' Theorem, we have:

$$\Delta a:=a(T)-a(0)=\int_{0}^T\frac{\dot{t}(s)}{\lambda(s)}\,ds=-\iint_{D}\frac{d\lambda\wedge dt}{\lambda^2}=-A_{\mathrm{hyp}}(D),$$
where $A_{\mathrm{hyp}}(D)$ denotes the signed hyperbolic area enclosed by the curve $\widetilde{\gamma}:=\pi(\gamma)\subset \bH_{\C}^1$ with respect to. the Riemannian metric $g_{\mathrm{hyp}}$. Therefore, the height $\Delta a$ prescribes the signed hyperbolic area enclosed by $\widetilde{\gamma}$.
Since the horizontal length of $\gamma$ is equal to the hyperbolic length of its projection \eqref{lhyphor} and $\gamma$ is an extremal trajectory, then $\widetilde{\gamma}$ minimizes the hyperbolic length subject to a fixed enclosed hyperbolic area. In other words, $\widetilde{\gamma}$ is a solution of the isoperimetric problem on the hyperbolic plane. We recall that the isoperimetric inequality on the hyperbolic plane has the following form (by denoting $\mathbf{A}_{\mathrm{hyp}}$ the non-negative hyperbolic area):
\begin{equation*}
    \ell_{\mathrm{hyp}}^2\geq 4\pi\,\mathbf{A}_{\mathrm{hyp}}+\mathbf{A}_{\mathrm{hyp}}^2,
\end{equation*}
where the equality holds if and only if the curve is a hyperbolic circle (see \cite{Schmidt}). Hence, $\widetilde{\gamma}$ is identified with a hyperbolic circle $C_p(R)$ which passes through the projected identity element $\pi(\e)=(1,0)\in \bH^1_{\C}$:

$$(1,0)\in C_p(R)\iff d_{\mathrm{hyp}}\left((\lambda_0,t_0),(1,0) \right)=R.$$
Consequently, by defining the hyperbolic circle $$
\Gamma_R
:=
\left\{
p\in\bH_{\mathbb C}^{1}:
d_{\mathrm{hyp}}\bigl(p,(1,0)\bigr)=R
\right\},
$$
it follows that the centers of all hyperbolic circles of radius $R$
passing through $(1,0)\in \bH^1_{\C}$, belong to $\Gamma_R$.

We let the center $p$ of the hyperbolic circle $C_p(R)$ vary and consider as base the
hyperbolic circle $\Gamma_R$.  We notice that the points of $\Gamma_R$
are parametrized by
\begin{equation}\label{GammaR}
 \lambda_0(u)=\cosh R+\sinh R\cos u,
 \quad
 t_0(u)=\sinh R\sin u, \quad u\in(-\pi,\pi).
\end{equation}

Motivated by the isoperimetric interpretation, we keep fixed the hyperbolic radius $R$ and consider the horizontal lifts of the circles $C_p(R)$, with $p$ varying along $\Gamma_R$. Because all these circles have the same radius, they bound disks of the same hyperbolic area. Consequently, when equipped with a common orientation, Stokes’ theorem shows that their horizontal lifts undergo the same change in fiber height, after one complete hyperbolic circle. This construction mirrors the construction of the Pansu sphere (or bubble set) in the Heisenberg group, where horizontal lifts of Euclidean circles of fixed radius through the projected neutral element generate the corresponding surface. In our case, hyperbolic circles replace the Euclidean circles. The resulting family of lifts gives the surface introduced in the following:

\begin{prop}\label{prop: flask 1}
For every point $p=p(u)\in \Gamma_R,$ let $C_p(R)$ be the hyperbolic circle of radius $R$, centered at $p$. The horizontal lifts of this family of circles generate the surface patch
$$\mathcal{F}(\e,R):=\left\{\sigma_R(u,\psi) \, : (u,\psi)\in (-\pi,\pi)\times (-\pi,\pi)  \right\},$$
where its parametrization $\sigma_R(u,\psi)$ is given by

\begin{equation}\label{flask}
\begin{aligned}
 \lambda_R(u,\psi)
 &=\bigl(\cosh R+\sinh R\cos u\bigr)
   \bigl(\cosh R+\sinh R\cos\psi\bigr),\\
 t_R(u,\psi)
 &=\sinh R\sin u
  +\bigl(\cosh R+\sinh R\cos u\bigr)
   \sinh R\sin\psi,\\
 a_R(u,\psi)
 &=\psi-2\cosh R\,
   \arctan\!\left(e^{-R}\tan\frac\psi2\right),
\end{aligned}
\end{equation}
for all $(u,\psi)\in(-\pi,\pi)\times(-\pi,\pi)$. The surface constructed by the parametrization (\ref{flask}) shall be called the \textit{flask surface} $\mathcal{F}(\e,R)$. 
\end{prop}

\begin{proof}
For every $u\in (-\pi,\pi)$ we substitute the center $p(u)=\left(\lambda_0(u),t_0(u)\right)\in\Gamma_R$ given by (\ref{GammaR}) into the parametrization (\ref{hypcircle1}) of the hyperbolic circle $C_p(R)$. This substitution gives
\begin{align*}
    \lambda&=\bigl(\cosh R+\sinh R\cos u\bigr)
   \bigl(\cosh R+\sinh R\cos\psi\bigr)\\
    t&=\sinh R\sin u
  +\bigl(\cosh R+\sinh R\cos u\bigr)
   \sinh R\sin\psi.
\end{align*}

In our next step, we lift $C_p(R)$ to a horizontal curve in $\Aa$. The horizontality condition (\ref{horiz ode}) gives

\begin{equation*}
 \frac{da}{d\psi}
 =\frac{1}{\lambda}\frac{dt}{d\psi}
 =\frac{\sinh R\cos\psi}
 {\cosh R+\sinh R\cos\psi}=1-\frac{\cosh R}
 {\cosh R+\sinh R\cos\psi}.
\end{equation*}
A simple integration results in 
$
 a(\psi)=\psi-2\cosh R\,
 \arctan\!\left(e^{-R}\tan\frac\psi2\right)+A,
$
where $A$ is an integration constant. Choosing the normalization $A=0$, the third coordinate function becomes
$$a=\psi-2\cosh R\,
   \arctan\!\left(e^{-R}\tan\frac\psi2\right).$$
Combining these three coordinate functions leads exactly to the parametrization (\ref{flask}). Therefore, the union of the resulting horizontal lifts is the surface patch $\mathcal{F}(\e,R)$.
\end{proof}

\begin{rem}
The formulas for $\lambda_R$ and $t_R$ in \eqref{flask} remain well-defined at $\psi=\pm \pi$. The third coordinate function $a_R$ requires separate consideration since $\tan\frac{\psi}{2}$ is not defined at $\psi=\pm \pi$. Nevertheless, the following limits exist and satisfy:
\begin{equation*}
    \lim_{\psi\to -\pi^+}a_R(u,\psi)=\pi(\cosh R-1), \qquad \lim_{\psi\to \pi^-}a_R(u,\psi)=-\pi(\cosh R-1).
\end{equation*}
Accordingly, we set $a_R(u,\pm \pi)=\mp\pi(\cosh R-1)$, for all $u\in (-\pi,\pi)$, while $\sigma_R(-\pi,\psi)=\sigma_R(\pi,\psi)$, for all $\psi\in [-\pi,\pi]$. Therefore, we can extend the parametrization $\sigma_R$ of the flask surface to the square $[-\pi,\pi]\times[-\pi,\pi]$.
These facts come to an agreement with Stokes' theorem, because we obtain
$$\Delta a=-2\pi(\cosh R-1)=-A_{\mathrm{hyp}}(D_R),$$
where $A_{\mathrm{hyp}}(D_R)$ denotes the hyperbolic area of the hyperbolic disk $D_R$, bounded by $C_{p}(R)$.
\end{rem}

\begin{figure}[H]
\centering
\begin{subfigure}{0.45\textwidth}
    \includegraphics[width=1.\textwidth]{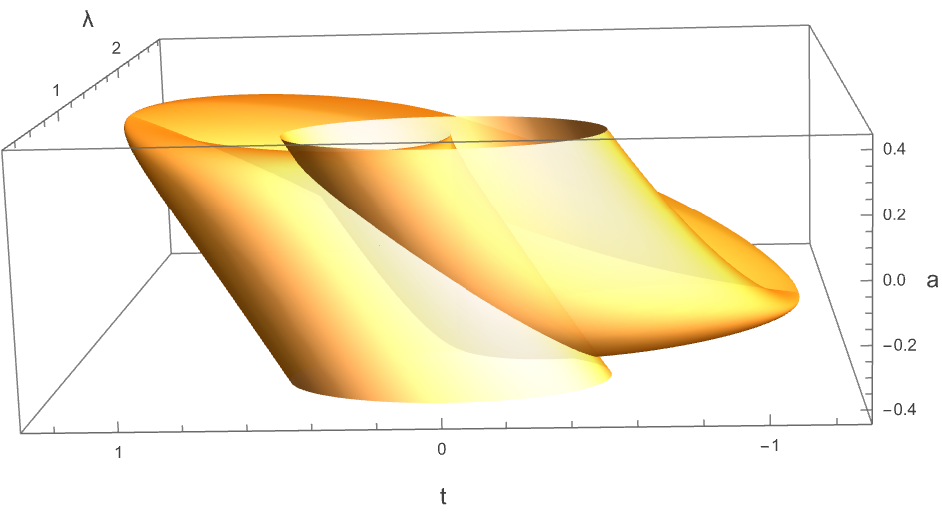}
\end{subfigure}
\begin{subfigure}{0.45\textwidth}
\includegraphics[width=1.\textwidth]{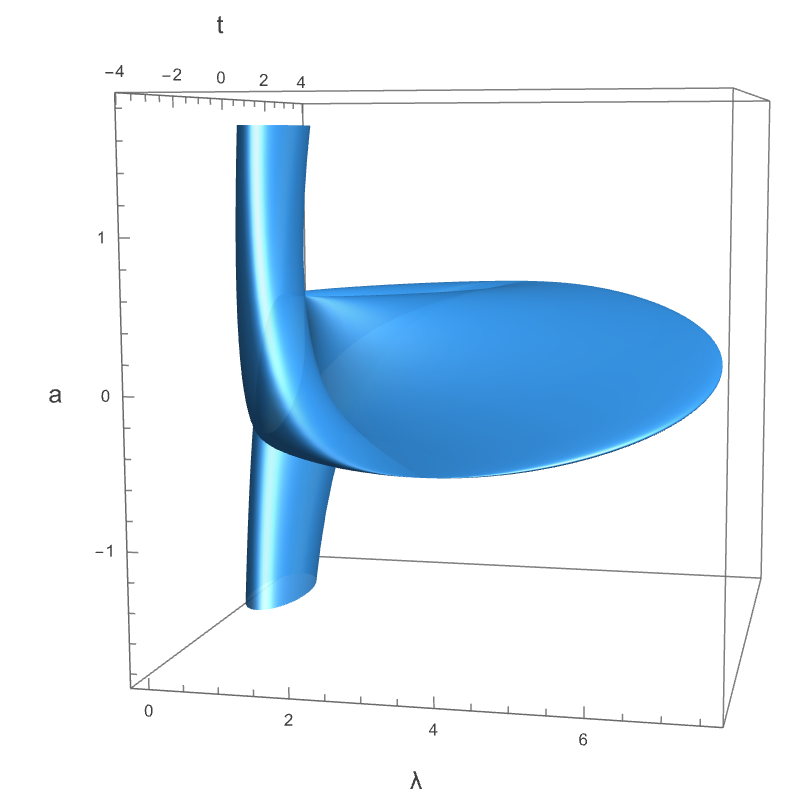}
\end{subfigure}
\begin{subfigure}{0.6\textwidth}
    \includegraphics[width=1.\textwidth]{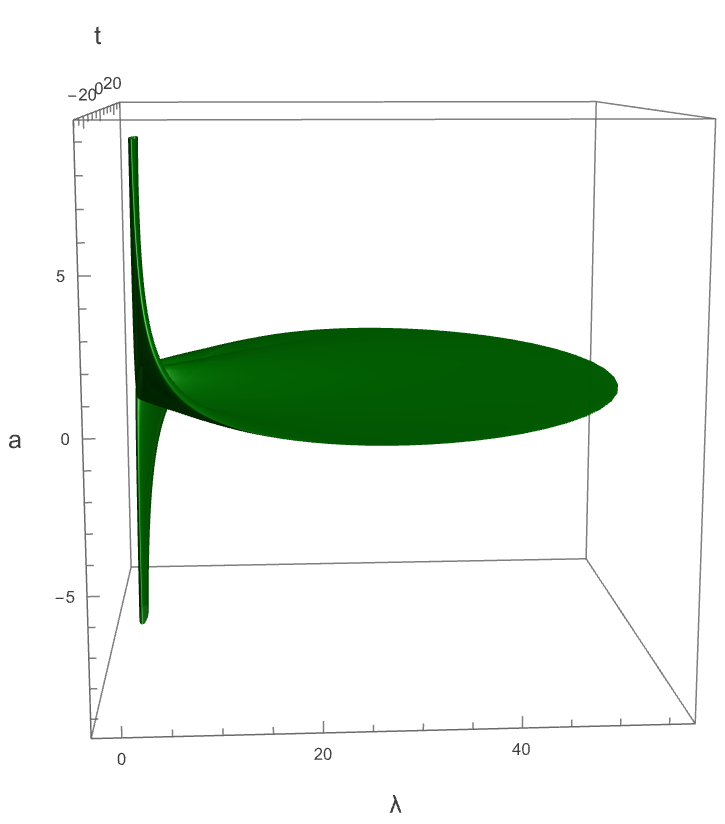}
\end{subfigure}
\caption{Plots of the flask surface $\mathcal{F}(\e,R)$ for $R=0.5$, $R=1$ and $R=2$, respectively.}
\end{figure}
Parametrizing horizontal lifts of hyperbolic circles by arc-length allows us to compare them with sub-Riemannian geodesics of type \eqref{AAGeod3} from Theorem \ref{thm: Geodesics of AA}. We are going to explicit this comparison by requiring a relation between the angular parameter of circles with the angular function appearing in this type of sub-Riemannian geodesics.

We introduce an arc-length $s$ for $C_p(R)$, normalized by the condition $s=0$, for $\psi=0$. Using (\ref{horiz speed}) we find that
\begin{equation}\frac{ds}{d\psi}
 =\frac{\sqrt{\lambda_\psi^2+t_\psi^2}}{\lambda(\psi)}=\frac{\sinh R}{\cosh R+\sinh R\cos\psi}.\label{dpsis}\end{equation}

Integrating (\ref{dpsis}), we obtain the arc-length
\begin{equation}s=2\sinh R \arctan \left( e^{-R} \, \tan\frac{\psi}{2}\right),\label{arclengthss}\end{equation}
where the inverse tangent is understood on the continuous branch determined by $s=0$ at $\psi=0$. Equivalently, the (inverse) angular function $\psi(s)$ satisfies
\begin{equation}
\dot{\psi}(s)=\coth R+\cos\psi(s), \quad \psi(0)=0. \label{psipsi}\end{equation}
 
An application of the chain rule together with (\ref{hypcircle1}) leads to $ \frac{1}{\lambda}\frac{d\lambda}{ds}
 =-\sin\psi,\
 \frac{1}{\lambda}\frac{dt}{ds}
 =\cos\psi$. Let $\widetilde{\gamma}_R(s)$ be a unit speed parametrization of the hyperbolic circle $C_p(R)$ determined by $\psi(0)=0$. To identify the horizontal lifts of these circles as left translations of the extremal trajectories described in Theorem \ref{thm: Geodesics of AA}, let $\bar{\gamma}_R(s)$ be such an extremal trajectory starting at $\e$ and let $\phi(s)$ be its corresponding angular function, satisfying \eqref{eq: angle ode}. Since
$
 \dot{\widetilde\gamma}_R
 =-\sin\psi\,f_1+\cos\psi\,f_2,
$
we compare this expression with its equivalent \eqref{Dpg}. Then, it must be $ \cos\varphi=-\sin\psi,
 \,
 \sin\varphi=\cos\psi.$
Consequently, we may choose
\begin{equation}\label{phipsi}
 \phi(s)=\psi(s)+\frac\pi2
 \quad (\mathrm{mod}\ 2\pi).
\end{equation}

Using (\ref{dpsis}) and (\ref{phipsi}), we obtain 
$$\dot{\phi}=\dot{\psi}=\coth R+\cos\psi=\coth R+\sin\phi.$$
Thus, by (\ref{kg}), follows that $k_{g}^{\widetilde\gamma_R}=\coth R$. By Proposition \ref{mom}, the corresponding momentum is $\xi_0=-\coth R$. Moreover, because $\psi(0)=0$, we have $\phi(0)=\frac{\pi}{2}.$

The following proposition completes the description of the flask as a surface generated by left translations of a specific sub-Riemannian geodesic of the form \eqref{AAGeod3}. Moreover, each generating curve is parametrized over an interval of length $2\pi \sinh R$, equal to the corresponding cut time.

\begin{prop}\label{Prop: flask 2}
Let $\psi(s)$ be a continuous solution of (\ref{psipsi}) and let $R>0$. Then, the flask surface $\mathcal{F}(\e,R)$ admits a parametrization $\Sigma_R(u,s)=\left(\lambda_R(u,s),t_R(u,s),a_R(u,s)\right)$, where
\begin{equation}\begin{aligned}
\lambda_R(u,s)
&=
\bigl(\cosh R+\sinh R\cos u\bigr)
\bigl(\cosh R+\sinh R\cos\psi(s)\bigr),\\
t_R(u,s)
&=
\sinh R\sin u
+
\bigl(\cosh R+\sinh R\cos u\bigr)
\sinh R\sin\psi(s),\\
a_R(u,s)
&=
\psi(s)-\coth Rs,\label{flasks}
\end{aligned}\end{equation}
for all $(u,s)\in [-\pi,\pi]\times[-\pi\sinh R,\pi\sinh R]$. 

Moreover, for every fixed $u\in [-\pi,\pi]$, the curve $\Sigma_R(u,\cdot) : [-\pi\sinh R,\pi\sinh R]\to \Aa$ is a sub-Riemannian geodesic.
\end{prop}

\begin{proof}
Substituting the arc-length form (\ref{arclengthss}) into the third coordinate function of (\ref{flask}) yields $a=\psi-\coth R\, s$. Furthermore, substituting the solution $\psi(s)$ of (\ref{psipsi}) back into the remaining coordinate functions of (\ref{flask}) completes the parametric equations (\ref{flasks}).

We express the generating curves of the flask with respect to the extremal trajectories of Theorem \ref{thm: Geodesics of AA}. For $s=0$, i.e. $\psi=0$, the curves start from the point
$$q(u):=\Sigma_R(u,0)=\left((\cosh R+\sinh R)\lambda_0(u),t_0(u),0 \right),$$
which is not, in general, the identity element $\e=(1,0,0)\in \Aa$. Hence, formula (\ref{AAGeod3}), which was derived for geodesics of $\Aa$ starting at $\e$, cannot be applied directly. 

We consider instead the geodesic starting at $\e$, with initial data
$\xi_0=-\coth R$ and $\phi_0=\pi/2$. Equivalently, its initial covector is  $(\eta_0,\tau_0,\xi_0)=\left(0, 1+\coth R,-\coth R\right)$.

By the formula (\ref{AAGeod3}) the corresponding trajectory $\bar{\gamma}_R(s)$, is described via the parametric equations:
\begin{align*}
    \bar{\lambda}_R(s)&=\left(\cosh R -\sinh R\right)\left( \cosh R+\sinh R\cos \psi(s)\right)\\
    \bar{t}_R(s)&=\left(\cosh R -\sinh R\right)\sinh R \sin \psi(s)\\
    \bar{a}_R(s)&= \psi(s)-\coth R\,s.
\end{align*}
Because $\psi(0)=0$, we have $\bar{\gamma}_R(0)=\e$. 

Since the sub-Riemannian structure is left-invariant, the curve 
$$\gamma_u(s)=q(u)\cdot \bar{\gamma}_R(s),$$
is also a unit-speed sub-Riemannian geodesic. Using the group law together with the elementary identity $\cosh^2 R-\sinh^2 R=1$, we obtain $\gamma_u(s)=:\Sigma_R(u,s)$.

Finally, since $\xi_0=-\coth R$, the cut time of each generating curve is
$$T_{\mathrm{cut}}=\frac{2\pi}{\sqrt{\xi_0^2-1}}=2\pi\sinh R,$$
while the interval $[-\pi\sinh R,\pi\sinh R]$ parametrizes each generating curve of the flask, which stops being a geodesic when its projection completes a full hyperbolic circle.
\end{proof}

\bibliographystyle{abbrv}
\bibliography{Library.bib}   

@article{BPPT,
 author = {E. Bubani and A. Pinamonti and I. D. Platis and D. Tsolis},
 title = {Horizontal curvatures of surfaces in {3D} contact sub-{Riemannian} {Lie} groups},
  JOURNAL = {Preprint },
 year = {Submitted, \url{https://arxiv.org/abs/2603.08330}, 2026},
 howpublished = {Preprint, {arXiv}:2603.08330 [math.{DG}] (2026)},
 url = {https://arxiv.org/abs/2603.08330},
 arXiv = {arXiv:2603.08330}
}

@article {Bubani,
    AUTHOR = {E.~Bubani},
     TITLE = {Hyperbolicity and quasiconformal maps on the affine-additive group},
   JOURNAL = {Doctoral Dissertation},
  FJOURNAL = {},
    VOLUME = {\url{https://boristheses.unibe.ch/6619/}},
      YEAR = {Universit\"at Bern, 2025},
    NUMBER = {},
     PAGES = {},
      ISSN = {},
   MRCLASS = {},
  MRNUMBER = {},
MRREVIEWER = {},
       DOI = {},
       URL = {},
}

@article{BBP1,
 author = {Balogh, Zolt{\'a}n M. and Bubani, Elia and Platis, Ioannis D.},
 title = {Hyperbolicity of the sub-{Riemannian} affine-additive group},
 fjournal = {Annales Fennici Mathematici},
 journal = {Ann. Fenn. Math.},
 issn = {2737-114X},
 volume = {50},
 number = {2},
 pages = {779--793},
 year = {2025},
 language = {English},
 doi = {10.54330/afm.178298},
 zbMATH = {8137312}
}

@incollection {P,
    AUTHOR = {I.D.~Platis},
     TITLE = {Quasiconformal mappings in the hyperbolic {H}eisenberg group
              and a lifting theorem},
 BOOKTITLE = {Lecture {N}otes of {S}eminario {I}nterdisciplinare di
              {M}atematica. {V}ol. {XV}},
    SERIES = {Lect. Notes Semin. Interdiscip. Mat.},
    VOLUME = {15},
     PAGES = {75--105},
 PUBLISHER = {Semin. Interdiscip. Mat. (S.I.M.), Potenza},
      YEAR = {2020},
      ISBN = {978-88-97478-26-3},
   MRCLASS = {30C65},
  MRNUMBER = {4395000},
MRREVIEWER = {Ilmari\ Kangasniemi},
}

@article {T,
    AUTHOR = {R.~Timsit},
     TITLE = {Geometric construction of quasiconformal mappings in the
              {H}eisenberg group},
   JOURNAL = {Conform. Geom. Dyn.},
  FJOURNAL = {Conformal Geometry and Dynamics. An Electronic Journal of the
              American Mathematical Society},
    VOLUME = {22},
      YEAR = {2018},
     PAGES = {99--140},
      ISSN = {1088-4173},
   MRCLASS = {30C65 (30C75 30L10 32V20)},
  MRNUMBER = {3845546},
MRREVIEWER = {Daniel\ Meyer},
       DOI = {10.1090/ecgd/323},
       URL = {https://doi.org/10.1090/ecgd/323},
}

@book {H,
    AUTHOR = {J.~Heinonen},
     TITLE = {Lectures on analysis on metric spaces},
    SERIES = {Universitext},
 PUBLISHER = {Springer-Verlag, New York},
      YEAR = {2001},
     PAGES = {x+140},
      ISBN = {0-387-95104-0},
   MRCLASS = {30C65 (28A75 28A78 46E35)},
  MRNUMBER = {1800917},
MRREVIEWER = {Christopher\ Bishop},
       DOI = {10.1007/978-1-4613-0131-8},
       URL = {https://doi.org/10.1007/978-1-4613-0131-8},
}

@book {A,
    AUTHOR = {T.~Aubin},
     TITLE = {A course in differential geometry},
    SERIES = {Graduate Studies in Mathematics},
    VOLUME = {27},
 PUBLISHER = {American Mathematical Society, Providence, RI},
      YEAR = {2001},
     PAGES = {xii+184},
      ISBN = {0-8218-2709-X},
   MRCLASS = {53-01 (58-01)},
  MRNUMBER = {1799532},
MRREVIEWER = {Liviu\ I.\ Nicolaescu},
       DOI = {10.1090/gsm/027},
       URL = {https://doi.org/10.1090/gsm/027},
}

@article {Z,
    AUTHOR = {V.A.~Zorich},
     TITLE = {Asymptotic geometry and conformal types of
              {C}arnot-{C}arath\'eodory spaces},
   JOURNAL = {Geom. Funct. Anal.},
  FJOURNAL = {Geometric and Functional Analysis},
    VOLUME = {9},
      YEAR = {1999},
    NUMBER = {2},
     PAGES = {393--411},
      ISSN = {1016-443X,1420-8970},
   MRCLASS = {53C17},
  MRNUMBER = {1692466},
       DOI = {10.1007/s000390050092},
       URL = {https://doi.org/10.1007/s000390050092},
}

@Book{CDPT,
 Author = {L.~Capogna and D.~Danielli and S.D.~Pauls and J.T.~Tyson},
 Title = {An introduction to the {Heisenberg} group and the sub-{Riemannian} isoperimetric problem},
 FSeries = {Progress in Mathematics},
 Series = {Prog. Math.},
 ISSN = {0743-1643},
 Volume = {259},
 ISBN = {978-3-7643-8132-5},
 Year = {2007},
 Publisher = {Basel: Birkh{\"a}user},
 Language = {English},
 zbMATH = {5083840},
 Zbl = {1138.53003}
}

@Misc{G,
 Author = {H.~Gr{\"o}tzsch},
 Title = {{\"U}ber einige {Extremalprobleme} der konformen {Abbildung}. {I}, {II}.},
 Year = {1928},
 Language = {German},
 HowPublished = {Berichte {Leipzig} 80; 367-376, 497-502 (1928).},
 zbMATH = {2576847},
 JFM = {54.0378.01}
}

@Book{Mont,
 Author = {R.~Montgomery},
 Title = {A tour of subriemannian geometries, their geodesics and applications},
 FSeries = {Mathematical Surveys and Monographs},
 Series = {Math. Surv. Monogr.},
 ISSN = {0076-5376},
 Volume = {91},
 ISBN = {0-8218-1391-9},
 Year = {2002},
 Publisher = {Providence, RI: American Mathematical Society (AMS)},
 Language = {English},
 zbMATH = {1731778},
 Zbl = {1044.53022}
}

@Article{CS-RT,
 Author = {G.~Citti and A.~Sarti},
 Title = {A cortical based model of perceptual completion in the roto-translation space},
 FJournal = {Journal of Mathematical Imaging and Vision},
 Journal = {J. Math. Imaging Vis.},
 ISSN = {0924-9907},
 Volume = {24},
 Number = {3},
 Pages = {307--326},
 Year = {2006},
 Language = {English},
 DOI = {10.1007/s10851-005-3630-2},
 URL = {amsacta.unibo.it/822/1/corticalpaper4.pdf},
 zbMATH = {5075566},
 Zbl = {1478.92100}
}

@book{D-MLG,
 author = {{Le Donne}, Enrico},
 title = {Metric {Lie} groups. {Carnot}-{Carath{\'e}odory} spaces from the homogeneous viewpoint},
 fseries = {Graduate Texts in Mathematics},
 series = {Grad. Texts Math.},
 issn = {0072-5285},
 volume = {306},
 isbn = {978-3-031-98831-8; 978-3-031-98834-9; 978-3-031-98832-5},
 year = {2025},
 publisher = {Cham: Springer},
 language = {English},
 doi = {10.1007/978-3-031-98832-5},
 zbMATH = {8074076}
}

@book{ABB,
 author = {A.~Agrachev and D.~Barilari and U.~Boscain},
 title = {A comprehensive introduction to sub-{Riemannian} geometry. {From} the {Hamiltonian} viewpoint. {With} an appendix by {Igor} {Zelenko}},
 fseries = {Cambridge Studies in Advanced Mathematics},
 series = {Camb. Stud. Adv. Math.},
 volume = {181},
 isbn = {978-1-108-47635-5; 978-1-108-67732-5},
 year = {2020},
 publisher = {Cambridge: Cambridge University Press},
 language = {English},
 doi = {10.1017/9781108677325},
 zbMATH = {7073879},
 Zbl = {1487.53001}
}

@misc{PBGM-PMP,
 author = {L.S.~Pontryagin and V.G.~Boltyanskii and R. V.~Gamkrelidze and E. F.~Mishchenko},
 title = {The mathematical theory of optimal processes. {Authorized} translation from the {Russian}. {Translator}: {K}. {N}. {Trirogoff}. {Editor}: {L}. {W}. {Neustadt}},
 year = {1962},
 language = {English},
 howpublished = {New {York} and {London}: {Interscience} {Publishers}, a division of {John} {Wiley} \& {Sons}, {Inc}., viii, 360 p. (1962).},
 zbMATH = {3167340},
 Zbl = {0102.32001}
}

@book{AS-CT,
 author = {A. A.~Agrachev and Y. L.~Sachkov },
 title = {Control theory from the geometric viewpoint.},
 fseries = {Encyclopaedia of Mathematical Sciences},
 series = {Encycl. Math. Sci.},
 issn = {0938-0396},
 volume = {87},
 isbn = {3-540-21019-9},
 year = {2004},
 publisher = {Berlin: Springer},
 language = {English},
 zbMATH = {2105696},
 Zbl = {1062.93001}
}

@book{Sachkov,
 author = {Y.~Sachkov},
 title = {Introduction to geometric control. {Translated} from the {Russian}},
 fseries = {Springer Optimization and Its Applications},
 series = {Springer Optim. Appl.},
 issn = {1931-6828},
 volume = {192},
 isbn = {978-3-031-02069-8; 978-3-031-02072-8; 978-3-031-02070-4},
 year = {2022},
 publisher = {Cham: Springer},
 language = {English},
 doi = {10.1007/978-3-031-02070-4},
 zbMATH = {7595725},
 Zbl = {1507.49001}
}

@misc{Maxstrata,
 author = {A.~ Da Silva and L.~ Grama and D.~ Duarte Novaes and M.~ Quispe Tusco},
 title = {Maxwell {Strata} in the sub-{Riemannian} problem on solvable, nonnilpotent regular three-dimensional {Lie} groups},
 year = {2026},
 howpublished = {Preprint, {arXiv}:2602.17782 [math.{OC}] (2026)},
 url = {https://arxiv.org/abs/2602.17782},
 arXiv = {arXiv:2602.17782}
}

@article {BR-CCmetrics,
    AUTHOR = {U.~Boscain and F.~Rossi},
     TITLE = {Invariant {C}arnot-{C}aratheodory metrics on {$S^3,\ {\rm
              SO}(3),\ {\rm SL}(2)$}, and lens spaces},
   JOURNAL = {SIAM J. Control Optim.},
  FJOURNAL = {SIAM Journal on Control and Optimization},
    VOLUME = {47},
      YEAR = {2008},
    NUMBER = {4},
     PAGES = {1851--1878},
      ISSN = {0363-0129,1095-7138},
   MRCLASS = {53C17 (22E30 49Q20)},
  MRNUMBER = {2421332},
MRREVIEWER = {R.\ Felipe\ Monroy-P\'erez},
       DOI = {10.1137/070703727},
       URL = {https://doi.org/10.1137/070703727},
}

@misc{Coxeter,
 author = {Coxeter, H. S. M.},
 title = {Introduction to geometry. 2nd ed},
 year = {1969},
 language = {English},
 howpublished = {New {York} etc.: {John} {Wiley} and {Sons}, {Inc}. {XVI}, 469 p. (1969).},
 zbMATH = {3289061},
 Zbl = {0181.48101}
}

@book{Hatcher,
 author = {Hatcher, Allen},
 title = {Algebraic topology},
 isbn = {0-521-79540-0},
 year = {2002},
 publisher = {Cambridge: Cambridge University Press},
 language = {English},
 zbMATH = {2103273},
 Zbl = {1044.55001}
}

@article{Monti-CCballs,
 author = {Monti, Roberto},
 title = {Some properties of {Carnot}-{Carath}\'eodory balls in the {Heisenberg} group},
 fjournal = {Atti della Accademia Nazionale dei Lincei. Classe di Scienze Fisiche, Matematiche e Naturali. Serie IX. Rendiconti Lincei. Matematica e Applicazioni},
 journal = {Atti Accad. Naz. Lincei, Cl. Sci. Fis. Mat. Nat., IX. Ser., Rend. Lincei, Mat. Appl.},
 issn = {1120-6330},
 volume = {11},
 number = {3},
 pages = {155--167},
 year = {2000},
 language = {English},
 zbMATH = {2217083},
 Zbl = {1197.53064}
}

@book{chic,
 author = {Chicone, Carmen},
 title = {Ordinary differential equations with applications},
 edition = {3rd revised and expanded},
 fseries = {Texts in Applied Mathematics},
 series = {Texts Appl. Math.},
 issn = {0939-2475},
 volume = {34},
 isbn = {978-3-031-51651-1; 978-3-031-51654-2; 978-3-031-51652-8},
 year = {2024},
 publisher = {Cham: Springer},
 language = {English},
 doi = {10.1007/978-3-031-51652-8},
 zbMATH = {7814197},
 Zbl = {1559.34002}
}

@article {BBDS-Finslerspheres,
    AUTHOR = {Barilari, Davide and Boscain, Ugo and Le Donne, Enrico and
              Sigalotti, Mario},
     TITLE = {Sub-{F}insler structures from the time-optimal control
              viewpoint for some nilpotent distributions},
   JOURNAL = {J. Dyn. Control Syst.},
  FJOURNAL = {Journal of Dynamical and Control Systems},
    VOLUME = {23},
      YEAR = {2017},
    NUMBER = {3},
     PAGES = {547--575},
      ISSN = {1079-2724,1573-8698},
   MRCLASS = {49J30 (22E25 53C17 53C60)},
  MRNUMBER = {3657277},
MRREVIEWER = {Daniel\ Massart},
       DOI = {10.1007/s10883-016-9341-8},
       URL = {https://doi.org/10.1007/s10883-016-9341-8},
}

@article{Hof-VP,
 author = {W.C.~Hoffman},
 title = {Higher visual perception as prolongation of the basic {Lie} transformation group},
 fjournal = {Mathematical Biosciences},
 journal = {Math. Biosci.},
 issn = {0025-5564},
 volume = {6},
 pages = {437--471},
 year = {1970},
 language = {English},
 doi = {10.1016/0025-5564(70)90079-9},
 zbMATH = {3382127},
 Zbl = {0241.92007}
}

@article{Citti-Cort,
 author = {G.~Cocci and D.~Barbieri and G.~Citti and A.~Sarti},
 title = {Cortical spatiotemporal dimensionality reduction for visual grouping},
 fjournal = {Neural Computation},
 journal = {Neural Comput.},
 issn = {0899-7667},
 volume = {27},
 number = {6},
 pages = {1252--1293},
 year = {2015},
 language = {English},
 doi = {10.1162/NECO_a_00738},
 url = {hdl.handle.net/10486/669324},
 zbMATH = {7062385},
 Zbl = {1414.92084}
}

@article{H-VC,
 author = {W.C.~Hoffman},
 title = {The visual cortex is a contact bundle},
 fjournal = {Applied Mathematics and Computation},
 journal = {Appl. Math. Comput.},
 issn = {0096-3003},
 volume = {32},
 number = {2-3},
 pages = {137--167},
 year = {1989},
 language = {English},
 doi = {10.1016/0096-3003(89)90091-X},
 zbMATH = {4108570},
 Zbl = {0676.92020}
}

@article{HP-RT,
 author = {R.K.~Hladky and S.D.~Pauls},
 title = {Minimal surfaces in the roto-translation group with applications to a neuro-biological image completion model},
 fjournal = {Journal of Mathematical Imaging and Vision},
 journal = {J. Math. Imaging Vis.},
 issn = {0924-9907},
 volume = {36},
 number = {1},
 pages = {1--27},
 year = {2010},
 language = {English},
 doi = {10.1007/s10851-009-0167-9},
 zbMATH = {5788106},
 Zbl = {1490.53047}
}

@Inbook{LSL-Robot,
author="J. P.~Laumond 
and S.~Sekhavat
and F.~Lamiraux",
title="Guidelines in nonholonomic motion planning for mobile robots",
bookTitle="Robot Motion Planning and Control",
year="1998",
publisher="Springer Berlin Heidelberg",
address="Berlin, Heidelberg",
pages="1--53",
isbn="978-3-540-40917-5",
doi="10.1007/BFb0036070",
url="https://doi.org/10.1007/BFb0036070"
}

@phdthesis{Boscain,
  title={ Motion Planning and Optimal Control for Quantum Mechanical Systems},
  author={U.~Boscain},
  url={},
  year={2006, \url{http://www.cmapx.polytechnique.fr/~boscain/HDR/new-habilitation.pdf}},
  type={Habilitation à diriger des recherches},
  school       = {Universit{\'e} de Bourgogne}
}

@incollection{P-Fin,
 author = {A.~Pascucci},
 title = {Kolmogorov equations in physics and in finance},
 booktitle = {Elliptic and parabolic problems. A special tribute to the work of Haim Brezis},
 isbn = {3-7643-7249-4},
 pages = {353--364},
 year = {2005},
 publisher = {Basel: Birkh{\"a}user},
 language = {English},
 zbMATH = {2225916},
 Zbl = {1085.35089}
}

@book{P-Fin2,
 author = {A.~Pascucci},
 title = {{PDE} and martingale methods in option pricing.},
 fseries = {Bocconi \& Springer Series},
 series = {Bocconi Springer Ser.},
 issn = {2039-1471},
 volume = {2},
 isbn = {978-88-470-1780-1; 978-88-470-1781-8},
 year = {2011},
 publisher = {Milano: Springer; Milano: Bocconi University Press},
 language = {English},
 doi = {10.1007/978-88-470-1781-8},
 zbMATH = {5768418},
 Zbl = {1214.91002}
}

@article {Pansu1,
    AUTHOR = {Pansu, Pierre},
     TITLE = {Une in\'egalit\'e{} isop\'erim\'etrique sur le groupe de
              {H}eisenberg},
   JOURNAL = {C. R. Acad. Sci. Paris S\'er. I Math.},
  FJOURNAL = {Comptes Rendus des S\'eances de l'Acad\'emie des Sciences.
              S\'erie I. Math\'ematique},
    VOLUME = {295},
      YEAR = {1982},
    NUMBER = {2},
     PAGES = {127--130},
      ISSN = {0249-6291},
   MRCLASS = {53C20 (22E25 53C65 58F17)},
  MRNUMBER = {676380},
}

@misc{Pansu2,
 author = {Pansu, P.},
 title = {An isoperimetric inequality on the {Heisenberg} group},
 year = {1983},
 language = {English},
 howpublished = {Differential geometry on homogeneous spaces, {Conf}. {Torino}/{Italy} 1983, {Rend}. {Semin}. {Mat}., {Torino}, {Fasc}. {Spec}., 159-174 (1983).},
 zbMATH = {4032543},
 Zbl = {0634.53034}
}

@article{Schmidt,
 author = {Schmidt, Erhard},
 title = {{\"U}ber die isoperimetrische {Aufgabe} im {{\(n\)}}-dimensionalen {Raum} konstanter negativer {Kr{\"u}mmung}. {I}: {Die} isoperimetrische {{\(n\)}} {Ungleichungen} in der hyperbolischen {Ebene} und f{\"u}r {Rotationsk{\"o}rper} im {{\(n\)}}-dimensionalen hyperbolischen {Raum}},
 fjournal = {Mathematische Zeitschrift},
 journal = {Math. Z.},
 issn = {0025-5874},
 volume = {46},
 pages = {204--230},
 year = {1940},
 language = {German},
 doi = {10.1007/BF01181439},
 url = {https://eudml.org/doc/168901},
 zbMATH = {3036741},
 Zbl = {0022.40302}
}

@book{Dragomir,
 author = {Dragomir, Sorin and Tomassini, Giuseppe},
 title = {Differential geometry and analysis on {CR} manifolds},
 fseries = {Progress in Mathematics},
 series = {Prog. Math.},
 issn = {0743-1643},
 volume = {246},
 isbn = {0-8176-4388-5},
 year = {2006},
 publisher = {Basel: Birkh{\"a}user},
 language = {English},
 zbMATH = {5019706},
 Zbl = {1099.32008}
}

@book{Goldman,
 author = {Goldman, William M.},
 title = {Complex hyperbolic geometry},
 fseries = {Oxford Mathematical Monographs},
 series = {Oxford Math. Monogr.},
 isbn = {0-19-853793-X},
 year = {1999},
 publisher = {Oxford: Clarendon Press},
 language = {English},
 zbMATH = {1270600},
 Zbl = {0939.32024}
}

@article {AH-Cartan,
    AUTHOR = {Ardentov, Andrei and Hakavuori, Eero},
     TITLE = {Cut time in the sub-{R}iemannian problem on the {C}artan
              group},
   JOURNAL = {ESAIM Control Optim. Calc. Var.},
  FJOURNAL = {ESAIM. Control, Optimisation and Calculus of Variations},
    VOLUME = {28},
      YEAR = {2022},
     PAGES = {Paper No. 12, 19},
      ISSN = {1292-8119,1262-3377},
   MRCLASS = {53C17 (22E25 49K15)},
  MRNUMBER = {4379606},
MRREVIEWER = {Andrea\ Pinamonti},
       DOI = {10.1051/cocv/2022006},
       URL = {https://doi.org/10.1051/cocv/2022006},
}

@article {AS-Engel,
    AUTHOR = {Ardentov, A. A. and Sachkov, Yu. L.},
     TITLE = {Cut time in sub-{R}iemannian problem on {E}ngel group},
   JOURNAL = {ESAIM Control Optim. Calc. Var.},
  FJOURNAL = {ESAIM. Control, Optimisation and Calculus of Variations},
    VOLUME = {21},
      YEAR = {2015},
    NUMBER = {4},
     PAGES = {958--988},
      ISSN = {1292-8119,1262-3377},
   MRCLASS = {53C17 (22E25)},
  MRNUMBER = {3395751},
MRREVIEWER = {Nicolas\ Juillet},
       DOI = {10.1051/cocv/2015027},
       URL = {https://doi.org/10.1051/cocv/2015027},
}
\Addresses

\end{document}